\documentclass[11pt]{article}
\usepackage[english,ruled,vlined]{algorithm2e}
\usepackage[noend]{algpseudocode}

\usepackage{pgfplots}
\pgfplotsset{compat=newest}
\usetikzlibrary{shapes}
\usetikzlibrary{plotmarks}

\usepackage{multirow}
\usepackage{listings}

\usepackage{amssymb}
\usepackage{amsfonts}
\usepackage{amsmath}
\usepackage{textcomp}
\usepackage{amsthm}
\usepackage{mathbbol}
\usepackage{stmaryrd}
\usepackage{mathrsfs}
\usepackage{graphicx}
\usepackage{amsbsy}
\usepackage{xcolor}
\usepackage{hyperref}
\usepackage{tikz}
\usepackage[normalem]{ulem}
\usepackage{float}
\usepackage{dsfont}

\usepackage{empheq}
\usepackage{cleveref}

\newtheorem{theorem}{Theorem}[section]
\newtheorem{lemma}[theorem]{Lemma}
\newtheorem{corollary}[theorem]{Corollary}
\newtheorem{definition}[theorem]{Definition}
\newtheorem{proposition}[theorem]{Proposition}

\newtheorem{remark}[theorem]{Remark}

\catcode`@=11
\catcode`@=12

\def\NN{\mathbb{N}}

\def\RR{\mathbb{R}}

\def\bs{\bigskip}

\def\eps{\varepsilon}
\def\bar#1{{\overline #1}}

\def\pa{\partial}

\def\calV{{\mathcal V}}
\def\1e{\mathds{1}}

\definecolor{mygreen}{rgb}{0,0.7,0}

\newcommand{\blue}[1]{\textcolor{black}{#1}}
\newcommand{\red}[1]{\textcolor{black}{#1}}

\definecolor{mygreen}{rgb}{0,0.4,0}
\newcommand{\michel}[1]{\textcolor{black}{#1}}

\begin{document}

\title{\sf Optimization of spatial control strategies for population replacement, application to {\it Wolbachia}}
\author{Michel Duprez\footnote{Inria, \'equipe Mimesis, Universit\'e de Strasbourg, Icube, 1 Place de l’H\^opital, Strasbourg, France ({\tt michel.duprez@inria.fr}).}
\and Romane H\'elie\footnote{IRMA, Universit\'e de Strasbourg, CNRS UMR 7501, Inria, 7 rue Ren\'e Descartes, 67084 Strasbourg, France ({\tt romane.helie@unistra.fr}).}
\and Yannick Privat\footnote{IRMA, Universit\'e de Strasbourg, CNRS UMR 7501, Inria, 7 rue Ren\'e Descartes, 67084 Strasbourg, France ({\tt yannick.privat@unistra.fr}).} 
\and Nicolas Vauchelet\footnote{Laboratoire Analyse, G\'eom\'etrie et Applications CNRS UMR 7539, Universit\'e Sorbonne Paris Nord, Villetaneuse, France ({\tt vauchelet@math.univ-paris13.fr}).}
}

\maketitle

\begin{center}
\textit{This article is dedicated to our friend Enrique Zuazua on the occasion of his 60\textsuperscript{th} birthday.}
\end{center}

\begin{abstract}
In this article, we are interested in the analysis and simulation of solutions to an optimal control problem motivated by population dynamics issues. In order to control the spread of mosquito-borne arboviruses, the population replacement technique consists in releasing into the environment mosquitoes infected with the {\it Wolbachia} bacterium, which greatly reduces the transmission of the virus to the humans. 
Spatial releases are then sought in such a way that the infected mosquito population invades the uninfected mosquito population. Assuming very high mosquito fecundity rates, we first introduce an asymptotic model on the proportion of infected mosquitoes and then an optimal control problem to determine the best spatial strategy to achieve these releases. We then analyze this problem, including the optimality of natural candidates and carry out first numerical simulations in one dimension of space to illustrate the relevance of our approach. 
\end{abstract}



\noindent {\bf Keywords: } reaction-diffusion equation, optimal control, second order optimality conditions.\\

\noindent {\bf 2010 AMS subject classifications: }92D25, 49K15, 65K10

\bs

\tableofcontents

\section{Introduction and state of the art}

\textit{Aedes} mosquitoes are the main vector of the transmission to human of many diseases, such as dengue, zika, or chikungunya. Since there are still no vaccines against these diseases, the best way to fight against them is to act on the vector population. Several techniques have been proposed. Some approaches aim at reducing the size of the population of mosquitoes. The use of insecticides is one of them, but its environmental consequences are too important to be used for a long time and on a large scale. The sterile insect technique (SIT) or the incompatible insect technique (IIT) are very promising strategies, consisting in massive releases of sterile or incompatible males, after mating with these males, the wild females will not produce viable eggs which should reduce the size of the populations (see e.g. \cite{SIT} and references therein). This method has already been implemented successfully on the field (see \cite{Bossinetal,Zheng}).
Other strategies are based on genetic manipulations like, for example, the release of insects carrying a dominant lethal (RIDL) \cite{Thomasetal,Heinrichetal,Fuetal}.

However, the suppression of one population of insects might have consequences on the environment. Then, other approaches aim at replacing the wild population of mosquitoes by another population inoffensive to human. One strategy under investigation consists in using the bacteria \textit{Wolbachia} taking advantage of phenomena called \textit{cytoplasmic incompatibility} (CI) and \textit{pathogen interference} (PI) \cite{Bourtzis, Sinkins}.
In key vector species such as \textit{Aedes aegypti}, if a male mosquito infected with \textit{Wolbachia} mates with a non-infected female, the embryos die early in development \cite{Wer.Wolbachia}. This is the so-called \textit{cytoplasmic incompatibility} (CI). Moreover, it has been observed that \textit{Aedes} mosquitoes infected with some \textit{Wolbachia} strains are not able to transmit viruses like dengue, chikungunya and zika \cite{Wal.wMel}, this is the \textit{pathogen interference} (PI).
Then, one may release mosquitoes artificially infected by \textit{Wolbachia} to mate with wild ones. Over time and if the releases are large and long enough, it can be expected that the majority of mosquitoes will carry Wolbachia, due to cytoplasmic incompatibility. As a result of PI, the mosquito population then has reduced vectorial competence.

In this paper, we focus on the \textit{Wolbachia} strategy and investigate the question of optimizing the spatial 
distribution of the releases.
Several mathematical models have been proposed for the \textit{Wolbachia} technique, see e.g. \cite{Farkas, Fenton, Schraiber, Hughes}.
In these papers, the authors model the time dynamics of the mosquitoes population.
Then, the question of optimizing the time of releases has been investigated e.g. in \cite{colombien,APSV2018, bliman2,MBE}.
However the spatial 
distribution of mosquitoes may have an impact on the success of the strategy. It is therefore relevant to add spatial dependence in mathematical models, which makes the study much more complicated.

In order to have a model simple enough to be tractable from a mathematical point of view, the authors in \cite{BarTur.Spatial} introduce a model focusing only on the proportion of \textit{Wolbachia}-infected mosquitoes, denoted $p$ in the sequel : 
$$
p:=\frac{n_{in}}{n_{in}+n_{un}}
$$ 
where $n_{in}$ is the density of \textit{Wolbachia}-infected mosquitoes and $n_{un}$ the density of uninfected mosquitoes. This quantity solves a scalar reaction-diffusion equation
$$
\frac{\partial p}{\partial t} - D\Delta p = f(p),
$$
where $D$ is a diffusion coefficient and $f$ is a bistable function\footnote{The wording ``{\bfseries bistable function}'' means that $f(0) = f(1) = 0$ and there exists $\theta \in (0, 1)$ such that $f(x)(x - \theta) < 0$ on $(0, 1) \setminus \{ \theta \}$ (in particular, one has necessarily $f(\theta)=0$ whenever $f$ is \blue{continuous})}.
For this model, the conditions to initiate the spatial spread are well-known \cite{Zubelli}.
It has been proved later in \cite{Stru2016} that this model may be rigorously derived from a more general system governing the dynamics of \textit{Wolbachia}-infected and \textit{Wolbachia}-uninfected mosquitoes by performing a large fecundity asymptotics.

In this study, we are investigating the question of the best spatial strategy for mosquito release, i.e., giving a certain amount of mosquitoes, we are trying to determine optimal locations to release them in order to ensure the invasion of the environment by \textit{Wolbachia}-infected mosquitoes.
If we denote $u$ the release function, then the above model is modified into
\begin{equation}\label{eq:pintro}
 \left\{
 \begin{aligned}
\displaystyle \frac{\partial p(t,x)}{\partial t} - D \Delta p(t,x) &= f(p(t,x))+u(t,x)g(p(t,x)), && \quad t\in (0,T), \quad x\in \Omega, \\
\partial_\nu p(t,x) &= 0, && \quad  t\in(0,T), ~x\in \partial\Omega,  \\
p(0,x) &= 0, &&  \quad x \in \Omega,
 \end{aligned}
\right.
\end{equation}
where $\Omega$ is an open bounded connected subset of $\RR^d$ with a regular boundary $\partial\Omega$. The function $g$ is positive and vanishes when $p=1$.
The derivation of \eqref{eq:pintro} will be detailed in Section \ref{sec:model}.

Let us summarize the main assumptions on $f$ and $g$ we will use in the sequel. 
\begin{equation}\label{assump:fg}\tag{$\mathcal{H}_{f,g}$}
\left\{\begin{array}{ll}
f\text{ is $C^2$ and of bistable type.} \\
\text{Denoting by $\theta$ the only root of $f$ in $(0,1)$, we assume that $f''(\cdot)>0$ on $(0,\theta)$.} \\
g\text{ is nonnegative, decreasing on }[0,1].\text{ Moreover, }g(1)=0.
\end{array}\right.
\end{equation} 

A first study was carried out in \cite{mosquiCemracs}, giving rise to the first very simple numerical experiments. In the present article, we seek to complete the results of this study, by analyzing qualitatively the solutions and by proposing adapted numerical strategies. Let us mention that a problem of the same nature has been investigated in \cite{miyaoka2019optimal}, mainly from a numerical point of view. Authors characterize optimal vaccination strategies to minimizes the costs associated with infections by the Zika virus and vaccines in the state of Rio Grande do Norte in Brazil.

In \cite{NadinToledo}, an optimal control problem close to the one investigated hereafter is tackled. The authors consider a population whose evolution is driven by a reaction-diffusion equation and look at determining initial data submitted to $L^1$ and $L^\infty$ constraints, maximizing the total size of the population. In our article, we choose to deal with a least square criterion instead of the average criterion considered in \cite{NadinToledo} (and more recently in \cite{Idriss2}). \blue{For reasons that will appear later, constant solutions are natural candidate to solve the considered optimal control problem. We show, as in \cite{NadinToledo}, that for certain families of parameters, the constant functions are local minimizers of the optimal control problem. On the other hand, we complete this first analysis and also manage to show for our model, that these same functions can be or not be global minimizers depending on the considered range of parameters. These results are also illustrated numerically.   }
Finally, it is worth mentioning that exact controllability issues for similar reaction-diffusion systems have been investigated in \cite{LeBalchRDCont,Idriss1}.

The outline of the paper is the following.
In Section \ref{sec:model}, we present the derivation of system \eqref{eq:pintro} and present the optimal control problem we are looking at.
Section \ref{sec:main} contains the main mathematical results of this paper.
Their proofs are given in Section \ref{sec:proofs}.
 More precisely, the rigorous derivation of system \eqref{eq:pintro} is explained in Section \ref{sec:modelreduc} and Section \ref{sec:OCPanalysis} is devoted to the mathematical study of the optimal control problem.
%
Finally, numerical illustrations with the description of the numerical algorithm are provided in Section \ref{sec:num}.

\section{Modelling}
\label{sec:model}

In the whole article, we will consider a given bounded connected open domain $\Omega$ of $\RR^d$ assumed to have a Lipschitz boundary. Let $T>0$ denote a fixed horizon of time.

\subsection{Model with two compartments}
In order to justify the introduced model on the proportion of \textit{Wolbachia}-infected mosquitoes, we first explain how to derive it.
Let us denote $n_{in}$ the density of infected mosquitoes and $n_{un}$ the density of uninfected mosquitoes. The dynamics of these quantities is governed by the reaction-diffusion system
\begin{subequations}
    \begin{empheq}[left=\empheqlbrace]{align}
&  (\pa_t  - D \Delta) n_{in} = (1-s_f) \frac{F_{un}}{\eps} n_{in} \left(1 - \frac{N}{K} \right) - \delta d_{un} n_{in} + u,  &\text{ in } (0,T)\times\Omega \label{eq:ni}\\
&  (\pa_t - D \Delta )n_{un} = \frac{F_{un}}{\eps} n_{un} \left(1 - s_h \frac{n_{in}}{N} \right) \left(1 - \frac{N}{K} \right) - d_{un} n_{un}, & \text{ in } (0,T)\times\Omega \label{eq:nu}\\
& N=n_{in} + n_{un} & \text{ in } (0,T)\times\Omega\nonumber \\
 &\pa_\nu n_{in} = 0, \quad \pa_\nu n_{un} = 0, &\text{ on } (0,T)\times\pa\Omega \nonumber
    \end{empheq}
\end{subequations}
complemented by initial conditions $n_{in}(t=0,x)=n_{in}^{\text{init}}(x)\michel{\geqslant 0}$, $n_{un}(t=0,x)=n_{un}^{\text{init}}(x)\michel{> 0}$,
where the following notations are used:
\begin{itemize}
  \item $u$: instantaneous releases of \textit{Wolbachia} infected mosquitoes. It is on this control that we will act upon. \red{At this step, we do not make the admissible space of controls precise, this will be done in what follows;}
  \item $d_{un}$, $d_{in}=\delta d_{un}$ with $\delta >1$: death rates, respectively for uninfected and  infected mosquitoes. We assume that \red{$d_{in}>d_{un}$} since \textit{Wolbachia} decreases lifespan;
  \item $F_{un}$, $F_{in}=(1-s_f)F_{un}$: net fecundity rates, respectively for uninfected and  infected mosquitoes. We assume that $F_{in}<F_{un}$ since \textit{Wolbachia} reduces fecundity;
  \item $\eps$ : parameter without dimension quantifying the fecundity, we assume $\eps\ll 1$ meaning that the fecundity is considered to be large;
  \item $s_h\in (0,1)$:  cytoplasmic incompatibility parameter (fraction of uninfected females' eggs fertilized by infected males which will not hatch). Formally, a proportion $1 - s_h$ of uninfected female's eggs fertilized by infected males actually hatch. Cytoplasmic incompatibility is perfect when $s_h = 1$;
  \item $K$: carrying capacity;
  \item $D$: dispersal coefficient. 
\end{itemize} 
All the constants above are assumed to be positive.
Existence and uniqueness of solutions for such reaction-diffusion system is by now well-known see e.g. \cite{Evans,Perthame}.
The equations driving the dynamics of $n_{in}$ and $n_{un}$ are bistable and monostable reaction-diffusion equations, respectively. Note that in the reaction term of \red{the second equation}, the term $- \frac{n_{in}}{n_{in}+n_{un}}$ stands for the vertical transition of the disease whereas the coefficient $s_h$ models that this vertical transmission may or not be perfect because of the cytoplasmic incompatibility.

In accordance with \cite{APSV2018}, we will assume moreover that the relation
\begin{equation}\label{rel:coef}
s_f + \delta - 1 < \delta s_h
\end{equation}
holds true. \red{It is notable that such a parameters choice is relevant since for {\it Wolbachia}-infected {\it Aedes} mosquitoes, and more precisely in the case of wMel strain, CI is almost perfect in these species-strain combination (see \cite{Dut.Lab}) meaning that $s_h$ is close to 1. Furthermore, such mosquitoes typically have a slightly reduced fecundity. In that particular case, one has $s_f\simeq 0.1$, $\delta\simeq 1.1$ and $s_h\simeq 0.9$ so that \eqref{rel:coef} holds true.}

To model optimal strategies with an adapted optimal control problem, it is convenient to introduce the {\it Wolbachia}-infected equilibrium $(n_{in,W}^*,0)$ \red{for the uncontrolled system,} defined by
\begin{equation}\label{eqn1Wstar}
(n_{in,W}^*,0):= \left(K(1-\frac{\eps\delta d_{un}}{F_{un}(1 - s_f)}),0\right),
\end{equation}
that is $(n_{in,W}^*,0)$ is a stationary solution of \eqref{eq:ni}--\eqref{eq:nu}.
A possible approach hence consists in looking for controls steering the system as close as possible to the target state $(n_{in,W}^*,0)$. In some sense, it stands for the research of a control strategy ensuring the persistence of infected mosquitoes at the time horizon $T$.

This leads to define the least squares functional $J_T$ given by
\red{\begin{equation}
  \label{eq:J}
  J_T(u) = \frac 12 \int_\Omega n_{un}(T,x)^2\, dx + \frac 12 \int_\Omega {(n_{in,W}^* - n_{in}(T,x))_+}^2\, dx,
\end{equation}}
where $(n_{in},n_{un})$ denotes the unique solution to the reaction-diffusion system \eqref{eq:ni}. \red{Here, we use the notation $x_+=\max\{x,0\}$. Observe that the presence of this maximum in the definition of $J_T$ does not induce non-differentiability since the mapping $x\mapsto x_+^2$ from $\RR$ to $\RR$ is $C^1$.}
\subsection{Reduction for large fecundity}\label{sec:asympModel}

When the fecundity is large compared to other parameters, it is relevant to consider the asymptotics $\eps\to 0$, which allows us to reduce system \eqref{eq:ni}--\eqref{eq:nu}. This reduction is inspired by \cite{APSV2018} where the authors consider a differential system.
We first explain formally how to reduce this system and state the main result, the rigorous approach is postponed to Section~\ref{sec:modelreduc}.
Since $n_{in}$ and $n_{un}$ will depend on $\eps$, we use the notation $n_{in}^\eps$ and $n_{un}^\eps$.

\paragraph{Formal reasoning.}
We investigate formally the limit as $\eps\to 0$ in the \eqref{eq:ni}--\eqref{eq:nu}. 
From \eqref{eq:ni}--\eqref{eq:nu}, we expect that $n_{in}^\eps+n_{un}^\eps = K + O(\eps)$. Then, we introduce the variables
$$
n^\eps = \frac 1\eps \left(1-\frac{n_{in}^\eps+n_{un}^\eps}{K}\right), \qquad
p^\eps = \frac{n_{in}^\eps}{n_{in}^\eps+n_{un}^\eps},
$$
where $p^\eps$ is the proportion of infected mosquitoes in the population.
Consider a sequence $(u^\eps)_{\eps>0}$ of controls. 
From straightforward computations from \eqref{eq:ni}--\eqref{eq:nu}, we deduce
\begin{align}
  \label{eq:n}
   & \pa_t n^\eps - D \Delta n^\eps = 
     -\frac{1-\eps n^\eps}{\eps}( F_{un} n^\eps (s_h (p^\eps)^2 - (s_f+s_h) p^\eps +1) - d_{un} ((\delta -1)p^\eps+1)) - \frac{u^\eps}{\eps K},  \\
  \label{eq:peps}
   & \pa_t p^\eps - D \Delta p^\eps + \frac{2\eps D}{1-\eps n^\eps} \nabla p^\eps\cdot\nabla n^\eps = p^\eps(1-p^\eps)(F_{un} n^\eps (s_h p^\eps - s_f) + (1-\delta) d_{un}) + \frac{u^\eps(1-p^\eps)}{K(1-\eps n^\eps)}.
\end{align}
Letting formally $\eps$ going to $0$, assuming that $(n^\eps, p^\eps, u^\eps)$ converges to $(n^0,p^0,u^0)$, we deduce from \eqref{eq:n} that the limit should satisfy the relation
\begin{equation}\label{eq:h}
n^0 = h(p^0,u^0) := \frac{d_{un}((\delta-1) p^0 + 1) - u^0/K}{F_{un}(s_h (p^0)^2 - (s_f+s_h) p^0 + 1)}.
\end{equation}
Then, passing into the limit in \eqref{eq:peps}, we deduce
$$
\pa_t p^0 - D \Delta p^0  = p^0 (1-p^0)(F_{un} n^0 (s_h p^0 - s_f) + (1-\delta) d_{un}) + \frac{u^0 (1-p^0)}{K}.
$$
Injecting \eqref{eq:h} into this latter equation, we obtain the scalar reaction-diffusion equation for the fraction of infected mosquitoes
\begin{equation}
  \label{eq:p0}
\left\{\begin{array}{ll}
\pa_t p^0 - D \Delta p^0 = f(p^0) + u^0 g(p^0) & \text{in }(0,T)\times \Omega\\
\partial_n p^0=0 & \text{on }(0,T)\times \pa\Omega
\end{array}\right.
\end{equation}
with
\begin{equation}
  \label{eq:fg}
  f(p) = \frac{\delta d_{un} s_h p(1-p)(p-\theta)}{s_h p^2 - (s_f+s_h) p + 1}, \qquad g(p) = \frac{(1-p)(1-s_h p)}{K(s_h p^2 - (s_f+s_h) p + 1)},
\end{equation}
where we use the notation $\theta = \frac{s_f + \delta - 1}{\delta s_h}$. 
Under the assumption \eqref{rel:coef} on the coefficients, we have $0<\theta<1$. Hence equation \eqref{eq:p0} for $u^0=0$ is a bistable reaction-diffusion equation.

\blue{
\begin{remark}\label{rmq:Hfg}
We claim that if the coefficients $s_f$ and $s_h$ satisfy 
\begin{equation}\label{cond1}
(s_f+s_h)^2<4s_h
\end{equation} 
and
\begin{equation}\label{cond2}
\delta(s_f+s_h-2)-s_f+1<0,
\end{equation} 
then the particular functions $f$ and $g$ given by \eqref{eq:fg} satisfy assumption~\eqref{assump:fg}.\\
Let us show it. Assuming that \eqref{cond1} holds true, we infer that $h(p):=s_h p^2 - (s_f+s_h) p + 1>0$ for all $p\in \RR$, which implies that $f$ is $C^2$ and bistable. Straightforward computations yield 
$$f''(p)=\frac{\delta d_{un}s_h\psi(p)}{2(1-s_f)h(p)^3},$$
where
$$
\psi(p)=p^{3} s_f(1-1/\delta) + p^{3} s_h - 3 p^{2} + \frac{1}{s_h}  + \frac{3 p^{2}}{\delta} - \frac{3 p}{\delta} + \frac{s_f}{\delta s_h} + \frac{1}{\delta} - \frac{1}{\delta s_h}.$$
Hence, $f''$ and $\psi$ share the same sign on $(0,1)$. 
Since $\delta>1$, $s_f\in(0,1)$ and $s_h\in(0,1)$,
one has  $\psi^{(3)}=6(\delta s_f+\delta s_h-s_f)/\delta\geq0$, i.e. $\psi''$ is increasing. 
We deduce that $\psi'$ is decreasing on a interval $(0,a)$ and increasing on $(a,1)$ with $a\in [0,1]$. Furthermore, one has $\psi'(0)=-3/\delta<0$ and $\psi'(1)<0$ if we assume \eqref{cond2}.
In that case, $\psi$ is decreasing on $(0,1)$. In addition
$\psi(0)=(\delta+s_f+s_h-1)/(\delta s_h)>0$,
thus, under \eqref{cond1}-\eqref{cond2}, we have $f''>0$ on $(0,\theta)$. We remark that $g'$ is negative on $(-1/\sqrt{s_h},1/\sqrt{s_h})$. Since $s_h\in(0,1)$, we deduce that $g$ decreases on $(0,1)$. Moreover $g(1)=0$ and therefore, $g$ is positive on $(0,1)$.
\end{remark}
}

  We introduce the notation $F$ for the antiderivative of $f$,
  $$
  F(p) = \int_0^p f(q)\,dq.
  $$ 
 \blue{In what follows, we will assume:
  \begin{equation}\label{eq:defthetac}
  \exists \theta_c \in (\theta,1)\quad \mid \quad F(\theta_c) = 0.
  \end{equation}
 This assumption is necessary to guarantee that invasion of the infected population may occur in space by local release. We will check that this assumption is satisfied for the particular choice of parameters we will consider for the numerical experiments in Remark~\ref{rk:assumpHfgver}. }


We consider system \eqref{eq:ni}--\eqref{eq:nu} with Neumann boundary conditions to model that the boundary acts as a barrier, and initial conditions satisfying
\begin{equation}
  \label{hyp:init}
\blue{  n_{in}^{\text{init},\eps} \in  L^\infty(\Omega), \quad 0\leq n_{in}^{\text{init},\eps}, \quad
  n_{un}^{\text{init},\eps} \in  L^\infty(\Omega), \quad 0< n_{un}^{\text{init},\eps}.}
\end{equation}
We assume also that the initial conditions are well-prepared, i.e.
\begin{equation}
  \label{eq:wellprepared}
  n_{in}^{\text{init},\eps} + n_{un}^{\text{init},\eps} = K + \eps K_0^\eps,\quad \text{ with }
  \|K_0^\eps\|_\infty \leq C.
\end{equation}
A typical example of initial conditions is when the system is as the Wolbachia-free equilibrium for which $n_{in}^{\text{init},\eps} = 0$ and $n_{un}^{\text{init},\eps} = K(1-\frac{\eps d_{un}}{F_{un}})$. In this case, assumption \eqref{eq:wellprepared} is obviously satisfied.

\paragraph{Convergence result.}
Following the ideas in \cite{Stru2016}, where a similar asymptotic limit is performed, we derive an asymptotic model on the proportion of infected mosquitoes, as the fecundity rates tend to $+\infty$. 
\begin{theorem}\label{TH}
  Under the assumptions \eqref{hyp:init}--\eqref{eq:wellprepared} on the initial data, let us assume moreover that the sequence $(u^\eps)$ converges towards $u^0$, weakly star in $L^\infty((0,T)\times\Omega)$ as $\eps\searrow 0$.
  Then, up to extraction of subsequences, the solution $(n^\eps,p^\eps)$ of \eqref{eq:n}--\eqref{eq:peps} converges towards $(n^0,p^0)$ as $\eps\to 0$, with \blue{$n^0\in  L^\infty((0,T)\times\Omega)$}, $p^0\in L^2(0,T;H^1(\Omega))$, and satisfying \eqref{eq:h} almost everywhere and \eqref{eq:p0}--\eqref{eq:fg} in the weak sense.
  More precisely, we have
  $$
  p^\eps \to p^0 \text{ strongly in } L^2((0,T)\times\Omega), \qquad
  n^\eps \rightharpoonup n^0 \text{weakly-star in } L^\infty((0,T)\times\Omega),
  $$
  where $p^0$ is solution to \eqref{eq:p0}   \end{theorem}
The proof of this theorem is postponed to Section~\ref{sec:modelreduc}.

\blue{Let us now define the least squares functional $J_T^\eps$ given by
\begin{equation}
  \label{eq:Jeps}
  J_T^\eps(u) = \frac 12 \int_\Omega n_{un}^\eps(T,x)^2\, dx + \frac 12 \int_\Omega {(n_{in,W}^* - n_{in}^\eps(T,x))_+}^2\, dx,
\end{equation}
where $(n_{in}^\eps,n_{un}^\eps)$ have been introduced at the very beginning of Section~\ref{sec:asympModel} and $n_{in,W}^*$ in \eqref{eqn1Wstar}.
As a corollary of the convergence result above, let us make the asymptotic behavior of the functional $J_T^\eps$ precise.
\begin{corollary}\label{cor:CVJTeps}
Under the assumptions \eqref{hyp:init}--\eqref{eq:wellprepared} on the initial data, let $(u^\eps)$ be a  sequence converging towards $u^0$, weakly star in $L^\infty((0,T)\times\Omega)$ as $\eps\searrow 0$.
Then, $ p^\eps(T,\cdot)$ converges towards $p^0(T,\cdot)$ strongly in $L^2(\Omega)$, and moreover, $J_T^\eps(u_\eps)$ converges towards $J_T^0(u^0)$ defined by
\begin{equation}\label{def:j022}
J^0_T(u) =K^2 \int_\Omega(1-p^0(T,x))^2\, dx,
\end{equation}
where $p$ denotes the solution of \eqref{eq:p0}, as $\eps\searrow 0$.
\end{corollary}
The proof of this result is postponed to Section~\ref{proof:cor:CVJTeps}.}

\blue{In what follows, we will rather deal with the proportion $p^0$ to model optimal releases strategies. The following section is dedicated to modeling issues about the optimal control problem we will deal with. }

\subsection{Toward an optimal control problem}\label{sec:OCP}
In this section, we will introduce an optimal control problem modeling optimal mosquito releases. For this purpose, we assume all fecundity rates large, which legitimates the use of the asymptotic model \eqref{eq:p0}--\eqref{eq:fg} introduced in Section~\ref{sec:asympModel}.. We will focus on time-pulsed releases, which will lead us to further simplify the problem.

\blue{In order not to cumulate all the difficulties related to the search for release distributions in time and space, we will suppose that one release, which is an impulse in time\footnote{We consider Dirac measures since at the time-level of the study (namely, some generations), the release can be considered as instantaneous.}, is done at the beginning of the experiment, i.e. $u(t,x)$ will be assimilated to a particular approximation of a Dirac impulse in time, namely $u_0(x) \delta_{\{t=0\}}$.\\
More precisely, we will consider as choice of release term, the function
$$
u^0(t,x)=\frac{1}{\eta}\mathds{1}_{[0,\eta]}(t)u_0(x),
$$
where $u_0\in L^\infty(\Omega)$ will be given. Making the change of variable $t=\tau \eta$, and introducing $\tilde{p}$ given by $\tilde{p}(\tau,x)=p^0(t,x)$, one gets from system \eqref{eq:p0} that $\tilde{p}$ solves
$$
\frac{\partial \tilde{p}}{\partial \tau} - \eta D\Delta \tilde{p} = \eta f(\tilde{p}) + u_0 g(\tilde{p}),
\qquad \tau \in [0,1], \ x\in \Omega.
$$
We now provide a purely formal argument to justify the optimal control problem we will deal with.
}
Letting formally $\eta$ go to $0$ and denoting, with a slight abuse of notation, still by $\tilde{p}$ the formal limit of the system above yields
\begin{equation}\label{eq:ptilde}
\frac{\partial \tilde{p}}{\partial \tau}(\tau,x) = u_0(x) g(\tilde{p}(\tau,x)), \qquad \tau \in [0,1], \ x\in \Omega.
\end{equation}
Let us denote $G$ the anti-derivative of $1/g$ vanishing at $0$, namely
$$
G(p) = \int_0^p \frac{dq}{g(q)}.
$$
Then, by a direct integration of \eqref{eq:ptilde} on $[0,1]$, we obtain
$$
G(\tilde{p}(1,x)) = G(\tilde{p}(0,x))+u_0(x), \qquad x\in \Omega .
$$
Hence we arrive at the system
\begin{equation}\label{eq:psimple}
 \left\{
 \begin{aligned}
\displaystyle \frac{\partial p}{\partial t} - D\Delta p &= f(p), &&\quad t\in (0,T), \quad x\in \Omega, \\
\partial_\nu p(t,x) &= 0, && \quad x\in \partial\Omega,  \\
p(0^+,\cdot) &= G^{-1}(u_0(\cdot)),
\end{aligned}
\right.
\end{equation}
where $f$ and $g$ are given by \eqref{eq:fg}.

\blue{According to \eqref{assump:fg}, $G(0) = 0$, $G(1^-) = +\infty$, $G$ is continuous in $[0, 1)$ and strictly increasing, $G^{-1}(u_0)$ is well defined and in $[0,1)$ for positive $u_0$. Moreover $0$ and $1$ are subsolution and uppersolution to \eqref{eq:psimple}, hence thanks to a standard comparison argument for parabolic systems, the solution $p$ to System~\eqref{eq:psimple} satisfies $0 \leq p(t,x) < 1$ for a.e. $t \in[0,T]$ and $x \in \Omega$ 
(see e.g. \cite{conway1977comparison}).
}


To take into account biological constraints on the release procedure, we will moreover assume that the release function is such that:
\begin{itemize}
\item the local release of mosquitoes is bounded : $0\leq u_0 \leq M$ a.e. in $\Omega$ with $M>0$;
\item the total number of used mosquitoes is bounded (production limitation), reading 
$$
\int_\Omega u_0(x)\,dx \leq C,
$$
with $C\in (0,MT)$. \red{Note that it is relevant to choose the parameter $C$ strictly lower than $MT$. In the converse case, it would mean that the choice $u_0(\cdot)=M$ is admissible, so that the local maximal number of mosquitoes can be released (almost) everywhere in $\Omega$. Since producing infected mosquitoes has an important cost, it is reasonable from a biological point a view to assume that such a release is not possible.}
\end{itemize}
This leads to introduce the admissible set $ \mathcal{V}_{C,M} $ given by
\begin{equation*}
\mathcal{V}_{C,M} = \left\{u_0\in L^\infty(\Omega), 0\leq u_0\leq M \text{ a.e. in }\Omega, \ \int_\Omega u_0(x)\,dx \leq C\right\}.
\end{equation*}

The goal is to be as near as possible to the equilibrium $p=1$ at time $T$. 
Let us denote (with a slight abuse of notation) by $J_T$, the least squares functional defined by
$$
J_T(u_0)=\frac12\int_\Omega (1-p(T,x))^2\,dx.
$$
Observe that coincides, up to a positive multiplicative constant, with the asymptotic functional $J^0_T$ given by \eqref{def:j022}. The optimization problem thus reads
\[\label{prob:reduced}
\boxed{\inf_{u_0\in \mathcal{V}_{C,M} } J_T(u_0)},
\tag{$\mathcal{P}_{\text{reduced}}$}
\]
where $p$ is the solution of \eqref{eq:psimple}.

From now on and without loss of generality, we will assume in what follows that the diffusion coefficient $D$ is equal to 1. 

\section{Main results}
\label{sec:main}
Constant solutions are natural candidates to solve Problem~\eqref{prob:reduced}. Indeed, it has been observed in \cite[Theorem 2.1]{mosquiCemracs} that in the very simple case where $f(\cdot)=0$ and $G:x\mapsto x$, Problem~\eqref{prob:reduced} has a unique solution $u_0$, which is constant and equal to $\min \big(1, M, \frac{C}{\lvert \Omega \rvert} \big)$. Furthermore, as stated in the following result, constant solutions equal to $M$ are optimal for a given range of the parameters.
We show moreover that, outside of this range, constant functions remain critical points and show that they are still local minimizers whenever $C$ is small enough. We also comment on the sharpness of this result by highlighting that for certain parameters, constant functions may not be global minimizers for Problem~\eqref{prob:reduced}.

According to Corollary~\ref{cor:0959}, it is enough to concentrate on the constant function equal to $C/|\Omega|$.

\begin{theorem}\label{prop1.1}
Let us assume that $f$ and $g$ satisfy \eqref{assump:fg}. Problem~\eqref{prob:reduced} has a solution.
\begin{itemize}
\item[(i)] For every $M \in (0,C/|\Omega |]$, the constant function $\bar u_M$
 equal to $M$ is the unique solution to Problem~\eqref{prob:reduced}.
\item[(ii)] Let us assume that $M|\Omega| >C$. The constant function $\bar u(\cdot)=C/|\Omega|$ is a critical point for Problem~\eqref{prob:reduced} (meaning that it satisfies the first order optimality conditions stated in Proposition~\ref{theo:1orderopt}). 

Furthermore, if $C\leq |\Omega|G(\theta)$, there exists $K_T>0$ such that for every $h\in L^2(\Omega)$, the second order differential of $J_T$ at $\bar u$ satisfies
$$
d^2J_T(\bar{u})(h,h) \geq K_T\Vert h \Vert_{L^2(\Omega)}^2
$$
and it follows that the function $\bar u$ is a local minimizer for Problem~\eqref{prob:reduced}.
\end{itemize}
\end{theorem}

Let us comment on the sharpness of Theorem~\ref{prop1.1}.
As will be emphasized hereafter and in Section~\ref{sec:num}, we do not expect that $\bar u$ solves Problem~\eqref{prob:reduced} for all values of $C\in (|\Omega|G(\theta), M|\Omega|)$. In some case, this will be confirmed numerically, by using $\bar u$ as starting point of optimization algorithms and obtain at convergence a nonconstant minimizer $\tilde u$ such that $J_T(\tilde u)<J_T(\bar u)$.  

  Actually, even in the case $C<\min\{G(\theta),M\} |\Omega|$, where we know from Theorem~\ref{prop1.1} that the constant solution $\bar{u}$ is a local minimizer, under some conditions on $|\Omega|$ and $C$, we may construct non constant initial date $u_0$ such that $J_T(u_0)<J_T(\bar{u})$, as stated below in Proposition~\ref{prop:cstpasoptim}.
  
Recalling that $\theta_c\in (\theta,1)$ is defined by $\int_0^{\theta_c} f(p)\,dp = 0$. We assume
\begin{equation}
  \label{eq:assumptionCLM}
  G(\theta_c)<M  \quad \text{ and } \quad C < |\Omega| G(\theta).  
\end{equation}
The following result shows that under some conditions on $\Omega$, $M$ and $C$, the constant function $\bar{u}$ is not a global minimum of the optimization problem~\eqref{prob:reduced}.
\begin{proposition}\label{prop:cstpasoptim1}
  Let us assume \eqref{eq:assumptionCLM} and that:
  \begin{itemize}
  \item $C$ is large enough;
  \item the inradius\footnote{In other words, the radius of the largest ball inscribed in $\Omega$.} of $\Omega$ is large enough; 
  \item $T$ is large enough.
 \end{itemize} Then the constant solution $\bar{u}$ is not a global minimum for Problem~\eqref{prob:reduced}.
\end{proposition}

A proof of this result is provided in Section~\ref{constant:nonglobal}.

\begin{remark}
  The conditions stated in Proposition \ref{prop:cstpasoptim1} are not sharp; the obtention of necessary and sufficient condition for constant solution to be a global minimizers seems to be intricate and we let it open.
  A related problem concerns the issue of finding sufficient and necessary conditions guaranteeing invasion in a bistable reaction-diffusion system that is, up to our knowledge still open, and we refer to \cite{NadinToledo} for partial answers in this direction.
\end{remark}

\begin{remark}
It is notable that, for the sets of parameters from \cite{data} below, the functions $f$ and $g$ satisfy \eqref{assump:fg}. Indeed, 
the function $f''$ vanishes once on $(0,1)$ and its root $z$ satisfies $\theta < z \simeq 0.466$, while $\theta_c \simeq 0.582$  (see Figure \ref{zeros_fonction_f}).

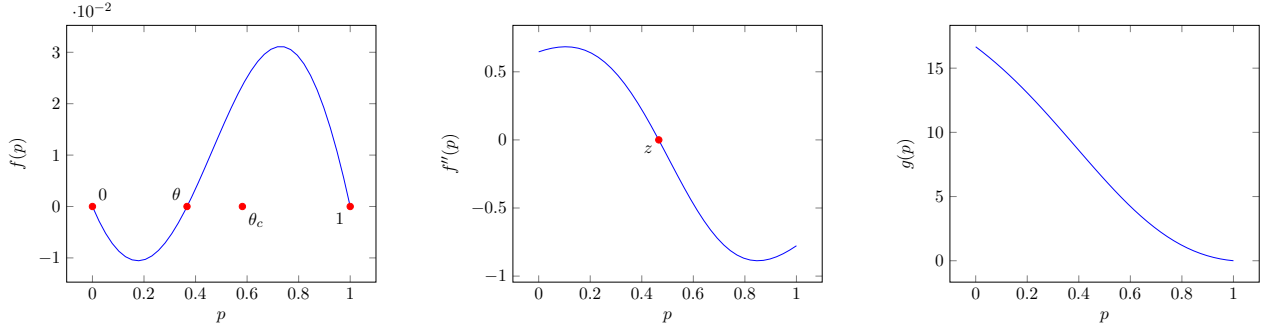
\begin{figure}[H] 
\begin{center}
\begin{tikzpicture}[thick,scale=0.6, every node/.style={scale=1.0}] \begin{axis}[xlabel=$p$,ylabel={$f(p)$},
xmin=-0.1,xmax=1.1,
legend style={at={(0.8,1.1)},
anchor=north west,
 legend columns=2}]
\addplot[color=blue]coordinates { 
(0.0,-0.0)
(0.02564102564102564,-0.0028352439305047365)
(0.05128205128205128,-0.0052357893728926)
(0.07692307692307693,-0.0071939703352266555)
(0.10256410256410256,-0.00870479334098737)
(0.1282051282051282,-0.009766349136427871)
(0.15384615384615385,-0.010380224712354674)
(0.1794871794871795,-0.010551903624802118)
(0.20512820512820512,-0.010291140626711095)
(0.23076923076923075,-0.009612294934917378)
(0.2564102564102564,-0.0085346052364522)
(0.28205128205128205,-0.007082388979023975)
(0.3076923076923077,-0.005285148780569625)
(0.3333333333333333,-0.003177570093457937)
(0.358974358974359,-0.0007993966817496234)
(0.3846153846153846,0.0018048259476426164)
(0.41025641025641024,0.004586144334255181)
(0.4358974358974359,0.0074917276268393505)
(0.4615384615384615,0.010465528192008012)
(0.48717948717948717,0.01344905897960561)
(0.5128205128205128,0.016382267966775003)
(0.5384615384615384,0.0192044842765304)
(0.5641025641025641,0.021855405980391273)
(0.5897435897435898,0.02427609655964444)
(0.6153846153846154,0.02640995581206602)
(0.641025641025641,0.028203631792776816)
(0.6666666666666666,0.029607843137254904)
(0.6923076923076923,0.030578085625584048)
(0.717948717948718,0.031075202752519057)
(0.7435897435897436,0.031065806899845855)
(0.7692307692307692,0.030522544939722864)
(0.7948717948717948,0.029424209203001924)
(0.8205128205128205,0.027755701249520388)
(0.8461538461538461,0.02550786139645981)
(0.8717948717948718,0.022677181235125653)
(0.8974358974358974,0.019265419265419276)
(0.923076923076923,0.015279141294892516)
(0.9487179487179487,0.010729207485723643)
(0.9743589743589743,0.0056302270679267526)
(1.0,0.0)
 };
 \filldraw[red] (0,0) circle[radius=2pt]  node[anchor=south west] {\color{black}$0$};
 \filldraw[red] (0.3668451,0) circle[radius=2pt]  node[anchor=south east] {\color{black}$\theta$};
 \filldraw[red] (1,0) circle[radius=2pt]  node[anchor=north east] {\color{black}$1$};
 \filldraw[red] ( 0.581719999,0) circle[radius=2pt]  node[anchor=north west] {\color{black}$\theta_c$}; 


\end{axis} 
\end{tikzpicture}\hfill
\begin{tikzpicture}[thick,scale=0.6, every node/.style={scale=1.0}] \begin{axis}[xlabel=$p$,ylabel={$f''(p)$},
xmin=-0.1,xmax=1.1,
legend style={at={(0.8,1.1)},
anchor=north west,
 legend columns=2}]
\addplot[color=blue]coordinates { 
(0.0,0.6463010204081633)
(0.02564102564102564,0.6614639442737119)
(0.05128205128205128,0.6731749368197163)
(0.07692307692307693,0.6808223661135449)
(0.10256410256410256,0.6837787682075737)
(0.1282051282051282,0.681415932993428)
(0.15384615384615385,0.673123195013805)
(0.1794871794871795,0.6583287447219854)
(0.20512820512820512,0.6365235302739232)
(0.23076923076923075,0.607287040227548)
(0.2564102564102564,0.5703139608661207)
(0.28205128205128205,0.5254404132808134)
(0.3076923076923077,0.4726682266377906)
(0.3333333333333333,0.41218553144665143)
(0.358974358974359,0.3443818969248599)
(0.3846153846153846,0.2698563218339291)
(0.41025641025641024,0.1894166399013586)
(0.4358974358974359,0.10406932409443599)
(0.4615384615384615,0.014999252563364296)
(0.48717948717948717,-0.07646030584002146)
(0.5128205128205128,-0.16886647694232976)
(0.5384615384615384,-0.26071259531863544)
(0.5641025641025641,-0.35047867276258315)
(0.5897435897435898,-0.43668369169951604)
(0.6153846153846154,-0.5179366749262163)
(0.641025641025641,-0.5929835319057432)
(0.6666666666666666,-0.660746997761042)
(0.6923076923076923,-0.7203575297601194)
(0.717948717948718,-0.7711737448199398)
(0.7435897435897436,-0.812791789084395)
(0.7692307692307692,-0.8450438393396462)
(0.7948717948717948,-0.8679866645567184)
(0.8205128205128205,-0.8818817596370052)
(0.8461538461538461,-0.8871689615798446)
(0.8717948717948718,-0.8844356559660406)
(0.8974358974358974,-0.8743836882961584)
(0.923076923076923,-0.8577959389832039)
(0.9487179487179487,-0.8355042436254928)
(0.9743589743589743,-0.80835998754461)
(1.0,-0.7772083200000002)
 };
 \filldraw[red] (0.466,0) circle[radius=2pt] node[anchor=north east] {\color{black}$z$};


\end{axis} 
\end{tikzpicture} \hfill
\begin{tikzpicture}[thick,scale=0.6, every node/.style={scale=1.0}] \begin{axis}[xlabel=$p$,ylabel={$g(p)$},
xmin=-0.1,xmax=1.1,
legend style={at={(0.8,1.1)},
anchor=north west,
 legend columns=2}]
\addplot[color=blue]coordinates { 
(0.0,16.666666666666668)
(0.02564102564102564,16.272169397801605)
(0.05128205128205128,15.857873634729712)
(0.07692307692307693,15.4238368387508)
(0.10256410256410256,14.970276931999422)
(0.1282051282051282,14.49758991472006)
(0.15384615384615385,14.00636653024102)
(0.1794871794871795,13.49740732141475)
(0.20512820512820512,12.971735354492344)
(0.23076923076923075,12.430605841177892)
(0.2564102564102564,11.875511875511876)
(0.28205128205128205,11.308185527521026)
(0.3076923076923077,10.730593607305936)
(0.3333333333333333,10.144927536231886)
(0.358974358974359,9.553586937641134)
(0.3846153846153846,8.959156785243742)
(0.41025641025641024,8.364378215362782)
(0.4358974358974359,7.772113406672033)
(0.4615384615384615,7.1853052404105915)
(0.48717948717948717,6.606932754095393)
(0.5128205128205128,6.039963669391463)
(0.5384615384615384,5.487305487305487)
(0.5641025641025641,4.9517567813580925)
(0.5897435897435898,4.435960367239342)
(0.6153846153846154,3.9423599782490477)
(0.641025641025641,3.473161930536762)
(0.6666666666666666,3.0303030303030307)
(0.6923076923076923,2.6154256738724313)
(0.717948717948718,2.229860743874493)
(0.7435897435897436,1.8746185369256252)
(0.7692307692307692,1.5503875968992253)
(0.7948717948717948,1.2575409975358995)
(0.8205128205128205,0.9961493386907754)
(0.8461538461538461,0.7659995058067707)
(0.8717948717948718,0.5666180994010037)
(0.8974358974358974,0.39729837107667915)
(0.923076923076923,0.2571294997662462)
(0.9487179487179487,0.14502709716815526)
(0.9743589743589743,0.05976393246675636)
(1.0,0.0)
 };


\end{axis} 
\end{tikzpicture}

\end{center}
\caption{
(from left to right) Graphs of the function $f$, its second order derivative $f''$ and the function $g$ by using the data from \cite{data} (see Table \ref{tab:value}).
\label{zeros_fonction_f}}
\end{figure}


\end{remark}

\section{Proofs}\label{sec:proofs}

This section will be devoted to prove Theorem \ref{TH} and \ref{prop1.1}.

\subsection{Model reduction}\label{sec:modelreduc}

In this section, we will give a proof of Theorem \ref{TH}, allowing us to reduce the system \eqref{eq:ni}--\eqref{eq:nu} to a scalar reaction-diffusion equation for the proportion as the parameter $\eps$ goes to $0$.  It is inspired by \cite{APSV2018} in which the authors \blue{use} a model composed by two differential equations.

\subsubsection{Uniform a priori estimates}

We first establish some uniform bounds with respect to $\eps>0$.

\begin{lemma}\label{lem:Linf}
  Assume the assumptions of Theorem \ref{TH} hold. Let $u^\eps$ be given in $\mathcal{V}_{C,M}$, $\eps\in (0,1)$ and $(p^\eps,n^\eps)$ be the unique solution of \eqref{eq:n}--\eqref{eq:peps}.
  Then,
  $$
  n^\eps \text{ is uniformly bounded in } L^\infty([0,T]\times \Omega),
  \text{ and } 0 \leq p^{\eps} \leq 1 \text{ on } [0,T]\times \Omega.
  $$
\end{lemma}
\begin{proof}
  By nonnegativity of $n_{in}^{\text{init}}$ and $n_{un}^{\text{init}}$, it is standard to deduce the nonnegativity of $n_{in}^\eps$ and $n_{un}^\eps$ (Indeed $0$ is a subsolution for \eqref{eq:ni} and for \eqref{eq:nu}, \blue{ see e.g. \cite{conway1977comparison}).
  Moreover, since $\mathbb{R}_+$ is invariant for the equation of $n_{un}^\eps$ and $n_{un}^{\text{init}}$ is non identically equal to zero, we deduce that $n_{un}^\eps>0$  on $\Omega\times (0,T]$ 
  (see e.g. \cite[th. 2]{weinberger1975invariant}).
Therefore, $p^\eps$ is well-defined on $[0,T]\times \Omega$ and satisfies by definition $0\leq p^\eps \leq 1$ on $[0,T]\times \Omega$.
}

Consider the function $h$ defined in \eqref{eq:h}. We remark that the denominator is positive. 
  Let $\tilde{K} := \max\{\max_{p\in[0,1]} h(p,0), \|n^{\text{init},\eps}\|_\infty\}$.  Let $\eps_0$ be such that 
  $\tilde{K}\leq \frac{1}{\eps_0}$ ($\Leftrightarrow 1-\eps_0\tilde{K}\geq0$), then, \blue{thanks to this choice, we have for $0<\eps\leq \eps_0$
  \begin{align*}
    & -\frac{1-\eps \tilde{K}}{\eps} \left(F_{un} \tilde{K} (s_h (p^\eps)^2 - (s_f+s_h) p^\eps + 1) - d_{un} ((\delta-1)p^\eps +1 ) \right) - \frac{u^\eps}{\eps K}  \\
    & \leq -\frac{1-\eps \tilde{K}}{\eps}  F_{un} (s_h (p^\eps)^2 - (s_f+s_h) p^\eps + 1) (\tilde{K}-h(p,0)) \\
    & \leq 0 = \partial_t \tilde{K} - D \Delta \tilde{K}.
  \end{align*}
  Hence,}
  we have that $\tilde{K}$ is a supersolution for \eqref{eq:n} for any $0<\eps\leq \eps_0$.
  Then $n^\eps\leq \tilde{K}$ for any $0<\eps\leq \eps_0$.

  By the same token, we have that the negative constant 
  $\min\{-\|n^{\text{init},\eps}\|_\infty, \min_{p\in[0,1]} h(p,M)\}$ 
  is a subsolution for \eqref{eq:n}.
  Thus $n^\eps$ is uniformly bounded from below. We deduce the uniform bound of $n^\eps$ in $L^\infty([0,T]\times\Omega)$.
\end{proof}
\blue{
\begin{lemma}\label{lem:nrj}
  Under above assumptions, for $\eps>0$ small enough, we have the uniform estimate
    \begin{equation}\label{boundpEp}
  \int_0^T \int_\Omega |\nabla p^\eps|^2 \,dx \leq \overline{C}
  \end{equation}
  and
    \begin{equation}\label{boundDn}
    \eps D \int_0^T \int_\Omega |\nabla n^\eps|^2\,dxdt \leq C_0,
  \end{equation}
  for some nonnegative constants $\overline{C}$ and $C_0$.
\end{lemma}
}
\begin{proof}
  On the one hand, multiplying equation \eqref{eq:n} by $\eps n^\eps$ and integrating on $\Omega$, we get
  \begin{align*}
    & \eps \frac{d}{dt} \int_\Omega |n^\eps|^2 \,dx + \eps D \int_\Omega |\nabla n^\eps|^2\,dx  \\
    & = - \int_\Omega (1-\eps n^\eps)n^\eps( F_{un} n^\eps (s_h (p^\eps)^2 - (s_f+s_h) p^\eps +1) - d_{un} ((\delta -1)p^\eps+1))\,dx  - \frac{1}{K} \int_\Omega u^\eps n^\eps\,dx.
  \end{align*}
  Since from Lemma \ref{lem:Linf}, we know that $n^\eps$ and $p^\eps$ are uniformly bounded in $L^\infty ([0,T]\times \Omega)$, we deduce \eqref{boundDn} after an integration in time.

  On the other hand, we fix $\eps_0>0$ small enough such that, for all $\eps\leqslant\eps_0$, we have $|n^\eps| \leq C_1 < \frac 1\eps$ on $[0,T]\times \Omega$ for some constant $C_1>0$ (which is always possible thanks to Lemma \ref{lem:Linf}).
  Then, we multiply by $p^\eps$ the equation satisfied by $p^\eps$ \eqref{eq:peps} and integrate over $\Omega$, we deduce
  \begin{align*}
    \frac 12 \frac{d}{dt} \int_\Omega (p^\eps)^2 \,dx + D \int_\Omega |\nabla p^\eps|^2 \,dx + 2\eps D \int_\Omega \frac{p^\eps}{1-\eps n^\eps} \nabla p^\eps \cdot \nabla n^\eps \,dx \\
    \leq C_3 + \frac{1}{K(1-\eps_0 C_1)} \int_\Omega u^\eps(t,x)\,dx
  \end{align*}
   for some nonnegative constant $C_3$.
  Then, using a Cauchy-Schwarz inequality, we get
  \begin{align*}
    \frac 12 \frac{d}{dt} \int_\Omega (p^\eps)^2 \,dx + D \int_\Omega |\nabla p^\eps|^2 \,dx \leq \ 
    & C_2 + \frac{2\eps D}{1 - \eps C_1} \left(\int_\Omega |\nabla p^\eps|^2 \,dx\right)^{1/2} \left( \int_\Omega |\nabla n^\eps|^2 \,dx \right)^{1/2}  \nonumber \\
    & + \frac{1}{K(1-\eps C_1)} \int_\Omega u^\eps(t,x)\,dx.
  \end{align*}
From \eqref{boundDn} and the well-known inequality $2ab\leq a^2+b^2$, we deduce after an integration in time
  \begin{align*}
   \frac 12 \int_\Omega (p^\eps)^2 \,dx + D \left(1-\frac{\eps }{1 - \eps C_0}\right)  \int_0^T\int_\Omega |\nabla p^\eps|^2 \,dxdt 
    \leq C_2 T + \frac{DC_0^{1/2}}{1 - \eps C_1} + \frac{MT|\Omega|}{K(1-\eps C_1)},
  \end{align*}
  where we recall that $u^\eps\in \mathcal{V}_{C,M}$ and $p\in[0,1]$.
  Taking $\eps$ small enough, we get the desired estimate.
\end{proof}

\subsubsection{Compactness result and proof of Theorem \ref{TH}}

We first recall the following compactness result (see \cite{Simon}).
\begin{lemma}[Aubin-Lions]\label{Aubin-Lions}
Let $T > 0, q \in (1, \infty)$, $(\psi_n)_n$ a bounded sequence in $L^q (0, T ; H)$, where $H$ is a Banach space. 
If $\psi_n$ is bounded in $L^q (0, T ; V)$ and $V$ compactly embeds in $H$, and if $(\partial_t \psi_n)_n$ is bounded in $L^q (0, T ; V')$ uniformly with respect to $n$, 
then $(\psi_n)_n$ is relatively compact in $L^q (0, T ; H)$.
\end{lemma}

\noindent {\bf Proof of Theorem \ref{TH}}
We split the proof of the Theorem into several steps.

\textit{Step 1. Compactness}.
We use Lemma \ref{Aubin-Lions} with $q=2$, $H=L^2(\Omega)$, $V=H^1(\Omega)\cap L^\infty(\Omega)$.
Then, the sequence $(p^\eps)$ is clearly bounded in $L^2(0,T;V)$ from Lemma \ref{lem:Linf} and \ref{lem:nrj}. The compact embedding of $V$ in $H$ is well-known from the Rellich-Kondrachov Theorem.
We are left to verify the bound on the time derivative: let $\phi\in V$, we denote $\langle\cdot,\cdot\rangle := \langle\cdot,\cdot\rangle_{V',V}$ the duality bracket. From equation \eqref{eq:peps}, we get
\begin{align*}
\int_0^T |\langle \pa_t p^\eps(t),\phi\rangle|^2\,dt
& = \int_0^T \left|\langle D \Delta p^\eps(t) - \frac{2\eps D}{1-\eps n^\eps(t)} \nabla p^\eps(t) \cdot \nabla n^\eps(t) - \psi^\eps(t) ,\phi\rangle  \right|^2 \,dt,
\end{align*}
where $\psi^\eps$ is the function defining the right hand side in equation \eqref{eq:peps}, which is uniformly bounded in $L^1((0,T)\times \Omega)\cap L^\infty((0,T)\times \Omega)$ as a direct consequence of Lemma \ref{lem:Linf}. Then,
\begin{align*}
\int_0^T |\langle \pa_t p^\eps,\phi\rangle|^2\,dt \leq \ 
& C_0 \|\nabla \phi\|_{L^2(\Omega)}^2 \int_0^T \|\nabla p^\eps\|_{L^2(\Omega)}^2 \,dt 
+ C_1 \eps \|\phi\|_{L^\infty(\Omega)}^2 \int_0^T \|\nabla p^\eps\|_{L^2(\Omega)}^2 \|\nabla n^\eps\|_{L^2(\Omega)}^2\,dt \\
& +  C_2 \|\phi\|^2_{L^2(\Omega)} \|\psi^\eps \|_{L^\infty((0,T)\times \Omega)}^2.
\end{align*}
Hence, we get the required bound from Lemma \ref{lem:Linf} and \ref{lem:nrj} and estimate \eqref{boundDn}.
We may apply Lemma \ref{Aubin-Lions} and deduce the relative strong compactness of $(p_\eps)$ in $L^2(0,T;L^2(\Omega))$.

Moreover, from the estimates in Lemma \ref{lem:Linf}, we deduce the relative weak-star compactness of the sequence $(n_\eps)$ in $L^\infty((0,T)\times\Omega)$.
Therefore, there exists $p^0\in L^2(0,T;H^1(\Omega))$ and $n^0\in L^\infty((0,T)\times\Omega)$ such that, up to extraction of subsequences, we have $p_\eps \to \bar{p}$ strongly in $L^2((0,T)\times\Omega)$ and a.e., $\nabla p_\eps \rightharpoonup \nabla\bar{p}$ weakly in $L^2((0,T)\times\Omega)$, and $n_\eps \rightharpoonup \bar{n}$ in $L^\infty((0,T)\times\Omega)$-weak$\star$.
\\

\textit{Step 2. Passing to the limit.}
We now pass to the limit in the weak formulation of equations \eqref{eq:n} and \eqref{eq:peps}.
From the weak formulation of \eqref{eq:n}, we deduce that for any test function $\phi\in C^\infty_c((0,T)\times\Omega)$, we have
\begin{align*}
  & \eps \int_0^T\int_\Omega (-n^\eps \pa_t \phi - D n^\eps \Delta \phi)\,dxdt \\
  & = -\int_0^T \int_\Omega \left( F_{un} n^\eps (s_h (p^\eps)^2 - (s_f+s_h) p^\eps +1) - d_{un} ((\delta -1)p^\eps+1)) - \frac{u^\eps}{K}\right)\phi \,dxdt  \\
  & \quad + \eps \int_0^T \int_\Omega n^\eps( F_{un} n^\eps (s_h (p^\eps)^2 - (s_f+s_h) p^\eps +1) - d_{un} ((\delta -1)p^\eps+1)) \phi \,dxdt.
\end{align*}
From the $L^\infty$-bound of Lemma \ref{lem:Linf}, we deduce that the term of the left hand side and the last term of the right hand side converge to $0$ as $\eps\to 0$.
For the first term of the right hand side, we may pass into the limit thanks to the weak convergence of $n^\eps$, the strong convergence of $p^\eps$, and the weak convergence of $u^\eps$. We obtain, for any $\phi\in C_c^\infty((0,T)\times\Omega)$,
$$
0 = -\int_0^T \int_\Omega \left( F_{un} n^0 (s_h (p^0)^2 - (s_f+s_h) p^0 +1) - d_{un} ((\delta -1) p^0+1)) - \frac{u^0}{K}\right)\phi \,dxdt.
$$
As a consequence \eqref{eq:h} is verified almost everywhere.

We are left to pass into the limit in the weak formulation of \eqref{eq:peps}.
Let $\phi\in C^\infty_c([0,T]\times\overline{\Omega})$, we have
\begin{align}
  & \int_0^T\int_\Omega (- p^\eps\pa_t \phi + D \nabla p^\eps\cdot \nabla \phi + \frac{2\eps D \phi}{1-\eps n^\eps} \nabla p^\eps\cdot\nabla n^\eps)\,dxdt  \nonumber \\
  & = \int_0^T \int_\Omega p^\eps(1-p^\eps)(F_{un} n^\eps (s_h p^\eps - s_f) + (1-\delta) d_{un})\phi\,dxdt + \int_0^T \int_\Omega\frac{u^\eps(1-p^\eps)}{K(1-\eps n^\eps)}\phi\,dxdt.
    \label{eqfin}
\end{align}
From the above convergence it is straightforward to pass into the limit into the first two terms of the left hand side. For the third term, we use estimate \eqref{boundDn}, and a Cauchy-Schwarz inequality to get
$$
\int_0^T\int_\Omega \frac{2\eps D \phi}{1-\eps n^\eps} \nabla p^\eps\cdot\nabla n^\eps\,dxdt \leq \frac{2 \sqrt{\eps} D \|\phi\|_{L^\infty}}{1-\eps \|n^\eps\|_{L^\infty}}\|\nabla p^\eps\|_{L^2} \sqrt{\blue{\overline{C}}} \to 0,
$$
as $\eps\to 0$, thanks to Lemma \ref{lem:Linf} and \ref{lem:nrj}.

We may pass into the limit for the first term of the right hand side of \eqref{eqfin} since $(p^\eps)$ converges strongly and a.e., and $(n^\eps)$ converges weakly.
Then, for the last term of the right hand side of \eqref{eqfin}, we verify
\begin{align}
& \left|\int_0^T \int_\Omega \left(\frac{u^\eps(1-p^\eps)}{K(1-\eps n^\eps)}
  - \frac{u^0 (1-p^0)}{K}\right)\phi \,dxdt\right|  \nonumber \\
  & = \left|\int_0^T \int_\Omega \left(\frac{(u^\eps-u^0)(1-p^0)+ u^\eps(p^0-p^\eps) + \eps u^0 (1-p^0) n^\eps}{K(1-\eps n^\eps)}\right) \phi \,dxdt\right| \nonumber \\
  & \leq \left|\int_0^T \int_\Omega \left(\frac{(u^\eps-u^0)(1-p^0)}{K(1-\eps n^\eps)}\right) \phi \,dxdt\right| +
    \frac{\|u^0\|_{L^\infty} \|\phi\|_{L^2}}{K(1-\eps\|n^\eps\|_{L^\infty})} \|p^0 - p^\eps\|_{L^2} + \eps \frac{\|u^\eps\|_{L^\infty} \|n^\eps\|_{L^\infty}\|\phi\|_{L^1}}{K(1-\eps\|n^\eps\|_{L^\infty})}.
    \label{eq1}
\end{align}
From the strong $L^2$ convergence of $(p^\eps)$ and the $L^\infty$ bounds in Lemma \ref{lem:Linf}, we deduce that the last two terms go to $0$ as $\eps\to 0$.
For the first term, we write
$$
\frac{(u^\eps-u^0)(1-p^0)}{K(1-\eps n^\eps)} = \frac{(u^\eps-u^0)(1-p^0)}{K} + \eps\frac{(u^\eps-u^0)(1-p^0)n^\eps}{K(1-\eps n^\eps)}.
$$
It is then straightforward to conclude the convergence towards $0$ of the first term of the right hand side of \eqref{eq1}.

Finally, passing into the limit $\eps\to 0$ into \eqref{eqfin}, we obtain
\begin{align*}
\int_0^T\int_\Omega (- p^0\pa_t \phi + D \nabla p^0\cdot \nabla \phi)\,dxdt =
&\int_0^T \int_\Omega p^0(1-p^0)(F_{un} n^0 (s_h p^0 - s_f) + (1-\delta) d_{un})\phi\,dxdt \\
&+ \int_0^T \int_\Omega\frac{u^0(1-p^0)}{K}\phi\,dxdt.
\end{align*}
We conclude by using the fact that $(n^0,p^0,u^0)$ verifies the relation \eqref{eq:h}. \\


\subsubsection{Proof of Corollary~\eqref{cor:CVJTeps}}\label{proof:cor:CVJTeps}

First, observe that $p^\eps(T,\cdot)$ converges weakly-star to $p^0(T,\cdot)$ in $L^2(\Omega)$. Indeed, the proof is standard. Consider the variational formulation \eqref{eqfin} on $p^\eps$ where test functions $\phi$ are now chosen in $C^\infty([0,T]\times C^\infty_c(\overline{\Omega}))$ instead of $C^\infty_c([0,T]\times\overline{\Omega})$. The variational formulation \eqref{eqfin} is then modified by the addition of
$$
\int_\Omega (p^\eps(T,\cdot)\phi(T,\cdot)-p^\eps(0,\cdot)\phi(0,\cdot))\, dx
$$  
in the left-hand side term. Since $p^\eps(T,\cdot)$ is bounded in $L^2(\Omega)$ it converges weakly to some limit in $L^2(\Omega)$ up to a subsequence. Passing to the limit in the variational formulation and using Theorem~\ref{TH} allows us to identify the closure point of $p^\eps(T,\cdot)$ as $p^0(T,\cdot)$. Finally, by uniqueness of the closure point, we infer that the whole sequence $p^\eps(T,\cdot)$ converges to $p^0(T,\cdot)$. Since the $L^2(\Omega)$-norm is lower semicontinuous for the weak-topology, we get that
 $$
  \liminf_{\eps\to 0} \int_\Omega \big(p^\eps(T,x)\big)^2\,dx \leq \int_\Omega \big(p^0(T,x)\big)^2\,dx.
  $$

  Let us multiply the equation \eqref{eq:peps} by $p^\eps$ and then, integrate it on $(0,T)\times\Omega$, we obtain
  \begin{align}
    \frac 12 \int_\Omega \big(p^\eps(T)\big)^2\,dx =\ 
    & \frac 12 \int_\Omega (p^{\text{init},\eps})^2\,dx
      - D \int_0^T \int_\Omega |\nabla p^\eps|^2\,dx
      + \int_0^T \int_\Omega \frac{u^\eps p^\eps(1-p^\eps)}{K(1-\eps n^\eps)}\,dx  \label{eq:limsupp} \\
    & + \int_0^T \int_\Omega (p^\eps)^2 (1-p^\eps)(F_u n^\eps (s_h p^\eps - s_f) + (1-\delta) d_u)\,dx  \nonumber \\
    & - 2\eps D \int_0^T \int_\Omega \frac{p^\eps}{1-\eps n^\eps} \nabla p^\eps\cdot \nabla n^\eps \,dx.  \nonumber
  \end{align}
  By assumption on the initial data, we have the convergence of the first term of the right hand side.
  For the second term of the right hand side, the weak convergence in $L^2((0,T),H^1(\Omega))$ guarantee that this term is upper semi-continuous, then
  $$
  \limsup_{\eps\to 0}\left( - D \int_0^T \int_\Omega |\nabla p^\eps|^2\,dx\right)
  \leq - D \int_0^T \int_\Omega |\nabla p^0|^2\,dx.
  $$
  Due to the strong convergence in $L^2((0,T)\times\Omega)$ of the sequence $(p^\eps)_\eps$ and using also the uniform bound of the sequence $(n^\eps)_\eps$ (see Lemma \ref{lem:Linf}), we deduce the convergence of the third term of the right hand side of \eqref{eq:limsupp}.
  Using also the strong convergence of $(p^\eps)_\eps$ and the weak convergence of $(n^\eps)_\eps$, we get the convergence of the fourth term in the right hand side of \eqref{eq:limsupp}.
  Finally, from a Cauchy-Schwarz inequality we have
  $$
  2\eps D \int_0^T \int_\Omega \frac{p^\eps}{1-\eps n^\eps} \nabla p^\eps\cdot \nabla n^\eps \,dx \leq C \eps^{\frac 12} \left(\int_0^T\int_\Omega |\nabla p^\eps|^2\,dxdt\right)^{1/2}  \left(\eps \int_0^T\int_\Omega |\nabla n^\eps|^2\,dxdt\right)^{1/2}.
  $$
  Thanks to the estimates in Lemma \ref{lem:nrj}, we get that this latter term goes to $0$ as $\eps\to 0$.
  Finally, we have proved that 
  $$
  \limsup_{\eps\to 0} \int_\Omega \big(p^\eps(T,x)\big)^2\,dx \leq \int_\Omega \big(p^0(T,x)\big)^2\,dx.
  $$

  It follows that $\Vert p^\eps(T,\cdot)\Vert_{L^2(\Omega)}$ converges to $\Vert p^0(T,\cdot)\Vert_{L^2(\Omega)}$ as $\eps\searrow 0$, and since $p^\eps(T,\cdot)$ converges to $p^0(T,\cdot)$ weakly in $L^2(\Omega)$, it follows that this convergence is in fact strong, whence the claim.
  
  It remains to investigate the convergence of $J_T^\eps(u^\eps)$ as $\eps\searrow 0$.
According to Theorem~\ref{TH} and its proof, one has $n_{in}^\eps(T,\cdot)+n_{un}^\eps(T,\cdot)= K-\eps Kn^\eps(T,\cdot)$ and therefore, $(n_{in}^\eps(T,\cdot)+n_{un}^\eps(T,\cdot))_{\eps>0}$ converges to $K$ in $L^\infty(\Omega)$.
Since $(p^\eps(T))_{\eps>0}$ converges to $p^0(T,\cdot)$ in $L^2(\Omega)$ as $\eps\searrow 0$, it follows that $(n_{in}^\eps(T,\cdot),n_{un}^\eps(T,\cdot)))_{\eps>0}$ converges to $(Kp^0(T,\cdot),K(1-p^0(T,\cdot)))$ in $L^2(\Omega)$. Since the {\it Wolbachia}-infected equilibrium $(n_{in,W}^*,0)$ converges to $(K,0)$ as $\varepsilon\searrow 0$, according to \eqref{eqn1Wstar}, by passing to the limit in \eqref{eq:J}, it follows that $J^\eps_T(u^\eps)$ converges, as $\eps \searrow 0$ to $J^0_T(u^0)$ given by 
$$
J^0_T(u^0) = \frac{K^2}{2} (1-p^0(T))^2+\frac{K^2}{2} (1-p^0(T))^2=K^2 (1-p^0(T))^2,
$$
where $p^0$ denotes the solution of \eqref{eq:p0}.


\subsection{Analysis of the optimal control problem (\ref{prob:reduced})}
\label{sec:OCPanalysis}

\subsubsection{Existence of an optimal control}
As a preliminary remark, note that existence of an optimal control has been shown in \cite[Theorem~1.1]{mosquiCemracs} in a more general setting. To make this article self-contained, we recall the argument hereafter. The analysis to follow is valid under the assumption~\eqref{assump:fg} on $f$ and $g$. It is not restricted to the particular choice of functions $f$ and $g$ given by \eqref{eq:fg}.
\begin{lemma}
Let $T>0$, $C>0$ and $M>0$. Problem~\eqref{prob:reduced} admits a solution $u^*_0$.
\end{lemma}
\begin{proof}
In what follows, we will denote by $p_{u_0}$ the solution to Problem~\eqref{eq:psimple} associated to the control choice $u_0$.
Let $(u^n_0)_{n\in \NN}\in (\mathcal{V}_{C,M})^\mathbb{N} $ be a minimizing sequence for Problem \eqref{prob:reduced}.
Notice that, since $u_0^n$ belongs to $\mathcal{V}_{C,M}$ and the range of $G^{-1}$ is included in $[0,1[$, we infer from the maximum principle that $0 \leq p_{u_0^n}(t,\cdot) <1$ for a.e. $t\in [0,T]$ so that $(J_T (u_0^n))_{n\in \NN}$ is bounded and $\inf_{u_0 \in \mathcal{V}_{C,M}}J_T(u_0)$ is finite.

Since the class $\mathcal{V}_{C,M}$ is closed for the $L^\infty$ weak-star topology,  there exists $u_0^\infty\in \mathcal{V}_{C,M}$ such that, up to a subsequence, $u_0^n$ converges weakly-star to $u_0^\infty$ in $L^\infty$. Here and in the sequel, we will denote similarly with a slight abuse of notation a given sequence and any subsequence.

Multiplying the main equation of \eqref{eq:psimple} by $p_{u_0^n}$ and integrating by parts, we infer from the above estimates the existence of a positive constant $C$ such that
\[
\frac{1}{2}\int_0^T \int_\Omega \partial_t (p_{u_0^n}(t,x)^2) dx dt + D\int_0^T \int_\Omega |\nabla p_{u_0^n}(t,x)|^2 dx dt \leq C
\]
for every $n\in \mathbb{N}$, which also reads
\[
\frac{1}{2} \int_\Omega \left[(p_{u_0^n}(t,x))^2)\right]_{t=0}^{t=T}\, dx + D\int_0^T \int_\Omega |\nabla p_{u_0^n}(t,x)|^2 \, dx dt \leq C
\]
for every $n\in \mathbb{N}$.

By using the pointwise bounds on $p_{u_0^n}$, one gets that $p_{u_0^n}$ is uniformly bounded in $L^2(0,T;H^1(\Omega))$. Furthermore, according to \eqref{eq:psimple}, the sequence $\partial_t p_{u_0^n}$ is uniformly bounded in $L^2(0,T;W^{-1,1}(\Omega))$.
The Aubin-Lions theorem (see \cite{Simon} and Lemma~\ref{Aubin-Lions}) yields that $p_{u_0^n}$ converges (up to a subsequence) to $p^\infty \in L^2(0,T;H^1(\Omega))$, strongly in $L^2(0, T; L^2(\Omega))$ and weakly in $L^2(0, T; H^1(\Omega))$. 
Furthermore, using that the sequence $\partial_t p_{u_0^n}$ is uniformly bounded in $L^2(0,T;W^{-1,1}(\Omega))$ also yields that $\partial_t p^\infty$ belongs to $L^2(0,T;W^{-1,1}(\Omega))$.
Furthermore, reproducing the standard variational argument used in the proof of Corollary~\ref{proof:cor:CVJTeps} to show the weak convergence of $p^\eps(T,\cdot)$ to $p^0(T,\cdot)$ in $L^2(\Omega)$ as $\eps\searrow 0$, one shows that for all $t\in [0,T]$, $p_{u_0^n}(t,\cdot)$ also converges weakly, up to a subsequence, to $p^\infty(t,\cdot)$ in $L^2(\Omega)$.

Passing to the limit in \eqref{eq:psimple} yields that $p^\infty$ is a weak solution to
\begin{equation*}
 \left\{
 \begin{aligned}
\displaystyle \partial_t p^\infty(t,x) - D\Delta p^\infty(t,x) &= f(p^\infty(t,x)), &&\quad t\in (0,T), \quad x\in \Omega, \\
\partial_\nu p^\infty (t,x) &= 0, && \quad t\in (0,T), \quad x\in \partial\Omega.  \\
 \end{aligned}
\right.
\end{equation*}

It is standard that any solution to this bistable reaction-diffusion equation is continuous in time.

It remains to show that $u^\infty_0= G(p^\infty(0^+, .))$. Note first that $G$ is convex \blue{since $g$ is decreasing on $(0,1)$ under assumption \eqref{assump:fg}}. 
According to the convergence results above, since $p_{u_0^n}(0,\cdot)$ converges weakly  (up to a subsequence) to $p^\infty(0,\cdot)$ in $L^2(\Omega)$ and since $u_0^n := G(p_{u_0^n}(0^+, .))$, we get that $G(p^\infty(0,\cdot))=u_0^\infty$ and hence, $p^\infty=p_{u_0^\infty}$ by passing to the limit as $n\to +\infty$ in the variational formulation on $p_{u_0^n}$.

Finally, let us show that $u^\infty_0$ belongs to $\mathcal{V}_{C,M}$.
Since the derivative of $G$ is $1/g$ which is positive, $G$ is increasing and therefore, one has $0\leq u^\infty_0 \leq  M$ a.e. in $\Omega$. Moreover, since $\int_\Omega u_0^n \leq C$ rewrites $\langle u_0^n,1\rangle_{L^\infty,L^1}\leq C$, we immediately get that the integral condition is satisfied by $u_0^\infty$.

Therefore, $u^\infty_0$ solves Problem \eqref{prob:reduced}.
\end{proof}
\subsubsection{First and second order optimality conditions}

We now state first and second order optimality conditions. \red{The objective is twofold:} first, we will analyze the optimality of constant solutions, and second, we will use them to derive adapted numerical algorithms. 

\begin{definition}
Let $u_0\in  \mathcal{V}_{C,M} $. A function $h$ in $L^{\infty}(\Omega)$ is said to be an \textbf{admissible perturbation} of $u_0$ in $\mathcal{V}_{C,M}$ if, for every sequence of positive real numbers $(\varepsilon_n)_{n\in\NN}$ decreasing to 0, there exists a sequence of functions $h^n$ converging to $h$ for the weak-star topology of \red{$L^{\infty}(\Omega)$} as $n \rightarrow +\infty$, and such that $u+\varepsilon_n h^n \in \mathcal{V}_{C,M}$ for every $n \in \mathbb{N}$.
\end{definition}

\begin{proposition}\label{prop:adjoint}
Let $u_0 \in \mathcal{V}_{C,M}$ and $h$ be an admissible perturbation. The functional $J_T$ is two times differentiable in the sense of Fr\'echet at $u_0$ and one has
$$
dJ_T(u_0)\cdot h=\int_{\Omega} q(0^+,x)(G^{-1})' (u_0(x))h(x) \, dx,
$$
where $q$ denotes the adjoint state, solving the backward p.d.e. 
\begin{equation}\label{eq:q}
\left\lbrace
\begin{array}{lll}
- \partial_t q (t,x) - D\Delta q (t,x) = f'(p(t,x))q(t,x),                      
 \quad &(t,x) \in (0,T) \times \Omega, 
 \\
\partial_\nu q(t,x)= 0,                                  
\quad &(t,x) \in (0,T) \times \partial \Omega,\\
q (T,x)=p(T,x) - 1
\quad &x \in \Omega
\end{array}\right.
\end{equation}
and $p$ denotes the solution to \eqref{eq:psimple} associated to the control choice $u_0$.

Furthermore, the second order derivative of $J_T$ at $u_0$ reads
\begin{eqnarray*}
d^2J_T(u_0)(h,h)&=&-\int_{\Omega}  \int_0^T \dot{p}(t,x)^2 f''(p(t,x))q(t,x) \, dt dx+\int_{\Omega}  \dot{p}(T,x)^2 \, dx\\
&& +\int_\Omega q(0^+,x)(G^{-1})''(u_0(x))h(x)^2\, dx
\end{eqnarray*}
for every admissible perturbation $h$, where $\dot{p}$ denotes the solution to the linear system
\begin{equation}\label{pdot}
\left\lbrace
\begin{array}{lll}
\partial_t \dot{p} (t,x) - D \Delta \dot{p}(t,x) = \dot{p}(t,x) f'(p(t,x)) & (t,x)\in (0,T)\times \Omega, \\
\partial_\nu \dot{p}(t,x)= 0 & (t,x)\in (0,T)\times \partial \Omega,                                      
\\
\dot{p}(0^+,x)= (G^{-1})'(u_0)h(x) & x\in \Omega.
\end{array}\right.
\end{equation}
\end{proposition}

\begin{proof}

As a preliminary remark, we claim that for any element $u$ of the set $\mathcal{V}_{C,M}$ and any admissible perturbation $h$, the mapping $ u\in \mathcal{V}_{C,M}\mapsto p\in L^2(0,T;H^1(\Omega))$, where $p_{u}$ denotes the unique weak solution of \eqref{eq:psimple}, is  differentiable in the sense of G\^ateaux at $u$ in direction $ h $. Indeed, proving such a property is standard in calculus of variations and rests upon the implicit function theorem.

Let $u\in \mathcal{V}_{C,M}$.
Let $h$ denote an admissible perturbation. 
Observing that $p_{u+\varepsilon h}$ solves the system
\begin{equation*}
\left\lbrace
\begin{array}{lll}
\partial_t p_{u+\varepsilon h}(t,x) - D\Delta p_{u+\varepsilon h}(t,x) = f(p_{u+\varepsilon h}(t,x)),  
 \quad \quad &(t,x) \in (0,T) \times \Omega, 
 \\
\partial_\nu p_{u+\varepsilon h}(t,x)= 0,                                                             
 \quad \quad &(t,x) \in (0,T) \times \partial \Omega,  
\\
p_{u+\varepsilon h} (0^+,x)=G^{-1} (u+\varepsilon h(x)),    
\quad &x\in \Omega.
\end{array}\right.
\end{equation*}
Let $\dot{p}$ denote the derivative of $\varepsilon \mapsto p_{}(u+\varepsilon h)$ at $\varepsilon=0$. Standard computations yield that $\dot{p}$ solves the linearized reaction-diffusion system
\begin{equation}\label{dot0}
\left\lbrace
\begin{array}{lll}
\partial_t \dot{p} (t,x) - D \Delta \dot{p}(t,x) = \dot{p}(t,x) f'(p_{u}(t,x)), 
 \quad \quad &(t,x) \in (0,T) \times \Omega, 
 \\
\partial_\nu \dot{p}(t,x)= 0,                                                                      
 \quad \quad &(t,x) \in (0,T) \times \partial \Omega,  
\\
\dot{p} (0^+,x)=(G^{-1})' (u(x)) h(x),
\quad &x\in \Omega.
\end{array}\right.
\end{equation}
Furthermore, according to the chain rule, one has 
\begin{equation*}
dJ_T(u)\cdot h =\lim_{\varepsilon \rightarrow 0} \frac{J_T(u+\varepsilon h)-J_T(u)}{\varepsilon}=
 \int_{\Omega}\dot{p} (T,x) \left(
p_{u} (T,x) - 1 \right) \, dx.
\end{equation*}

Let us multiply the main equation of \eqref{dot0}  by $q_{u}$, and integrate then two times by parts on $(0,T)\times \Omega$. One thus gets 
\begin{multline}\label{intdoti2}
\int_0^T \int_{\Omega}    \partial_t \dot{p} (t,x) q_{u}(t,x) ~dx dt
=\int_0^T \int_{\Omega}      D  \dot{p}(t,x) \Delta q_{u}(t,x) ~dx dt\\
+\int_0^T \int_{\Omega}     \dot{p}(t,x) f'(p_u(t,x)) q_{u}(t,x) ~dx dt.
\end{multline}
Similarly, let us multiply the main equation of \eqref{eq:q} by $\dot{p}$, and integrate then by parts on $(0,T)\times \Omega$. We obtain 
\begin{multline}\label{intq}
- \int_0^T \int_{\Omega}\dot{p}  (t,x)   \partial_t q_{u} (t,x)~ dx dt
=\int_0^T \int_{\Omega}D \dot{p} (t,x)   \Delta q_{u} (t,x)  ~ dx dt\\
+ \int_0^T \int_{\Omega}  \dot{p}  (t,x) f'(p_u(t,x))q_{u}(t,x) ~ dx dt.
\end{multline}
By comparing \eqref{intdoti2} and \eqref{intq}, we infer that
\begin{equation*}
 \int_0^T \int_{\Omega}\left(\dot{p}  (t,x)   \partial_t q_{u} (t,x)+ \partial_t \dot{p} (t,x) q_{u}(t,x) \right) \, dx dt=0
\end{equation*}
leading to the following duality identity:
\begin{equation*}
\int_{\Omega}  \left(  \dot{p} (T,x)  q_{u} (T,x)   -  \dot{p}  (0,x)  q_{u} (0,x)  \right) dx=0.
\end{equation*}
By using \eqref{dot0} and \eqref{eq:q}, we rewrite the expression above as
\begin{equation*}
\int_{\Omega}   \dot{p} (T,x)(  p_u(T,x) - 1 ) =\int_{\Omega}  (G^{-1})' (u(x)) h(x)  q_{u} (0,x)  dx.
\end{equation*}
Thus the desired expression of the derivative follows.

Let us now compute $d^2J_T(u_0)$.
Since $J_T$ is two times differentiable, one has
\begin{align*}
&d^2J_T(u_0)(h,h)\\&= 
\lim_{\varepsilon \rightarrow 0}\frac{dJ_T(u_0+\varepsilon h)\cdot h-dJ_T(u_0)\cdot h}{\varepsilon} \\
&=\lim_{\varepsilon \rightarrow 0}\frac{\int_{\Omega} q_{u_0+\varepsilon h}(0^+,x) (G^{-1})'(u_0+\varepsilon k)(x) h(x) \, dx-\int_{\Omega} q_{u_0}(0^+,x) (G^{-1})'(u_0(x))h(x)  ~dx}{\varepsilon} \\
&=\int_{\Omega} \dot{q}(0^+,x) (G^{-1})'(u_0(x))h(x)  \, dx+\int_\Omega q_{u_0}(0^+,x)h(x)^2(G^{-1})''(u_0(x))\, dx,
\end{align*}
where $\dot{q}$ is given by
$$
\dot{q}(t,x)=\lim_{\varepsilon \rightarrow 0}\frac{q_{u_0+\varepsilon h}(t,x)- q_{u_0}(t,x) }{\varepsilon}.
$$
A standard reasoning enables us to prove that $\dot{q}$ solves the linear p.d.e.  
\begin{equation}\label{qdot}
\left\lbrace
\begin{array}{lll}
- \partial_t \dot{q}(t,x) - D\Delta \dot{q} (t,x) =\dot{p}(t,x) f''(p_u(t,x))q(t,x)   + f'(p_u(t,x))\dot{q}(t,x) ,& (t,x)\in (0,T)\times \Omega, \\
\partial_\nu \dot{q} (t,x)= 0, & (t,x)\in (0,T)\times \partial \Omega, \\
\dot{q} (T,x)=\dot{p}(T,x), & x\in \Omega,
\end{array}\right.
\end{equation}
with $\dot{p}$, the solution of the linear p.d.e. \eqref{pdot}.
One has
\begin{align*}
&\int_{\Omega} \dot{q}(0^+,x)(G^{-1})'(u_0(x)) h(x)\, dx  \\
&=\int_{\Omega} \dot{q}(0^+,x) \dot{p}(0^+,x)  ~dx   
\\
&=\int_{\Omega} \dot{q}(0^+,x) \dot{p}(0^+,x)  ~dx  
-\int_{\Omega} \dot{q}(T,x) \dot{p}(T,x)  ~dx
+\int_{\Omega} \dot{q}(T,x) \dot{p}(T,x)  ~dx
 \\
&=\int_{\Omega}  \int_0^T \partial_t \dot{p} (t,x) \dot{q} (t,x) dt dx
+\int_{\Omega}  \int_0^T \partial_t \dot{q} (t,x) \dot{p} (t,x) dt dx
+\int_{\Omega}  \dot{p}(T,x)^2  ~ dx.
\end{align*}
By using the main equation in Systems~\eqref{qdot} and \eqref{pdot}, one gets
\begin{align*}
\int_{\Omega} \dot{q}(0^+,x)(G^{-1})'(u_0(x)) h(x)\, dx  &=\int_{\Omega}  \int_0^T \left[ D \Delta \dot{p}(t,x) + \dot{p}(t,x) f'(p_u(t,x))   \right] \dot{q} (t,x)\, dt dx
\\
&+\int_{\Omega}  \int_0^T  \left[  - D\Delta \dot{q} (t,x) -\dot{p}(t,x) f''(p_u(t,x))q_u(t,x)  \right.\\
&\left. - f'(p_u(t,x))\dot{q}(t,x)   \right] \dot{p} (t,x) \, dt dx +\int_{\Omega}  \dot{p}(T,x)^2  ~ dx.
\end{align*}
The Green formula finally yields
$$
\int_{\Omega} \dot{q}(0^+,x)(G^{-1})'(u_0(x)) h(x)\, dx  =\int_{\Omega}  \int_0^T   -\dot{p}(t,x)^2 f''(p_u(t,x))q_u(t,x)  dt dx
+\int_{\Omega}  \dot{p}(T,x)^2  ~ dx,
$$
whence the expected expression for the second order derivative.
\end{proof}

Let us now derive first and second order optimality conditions for this problem.

\begin{proposition}[\red{Necessary first and second orders optimality conditions}]\label{theo:1orderopt}
For all $u_0\in\mathcal{V}_{C,M}$ consider $\psi[u_0]$ denote the function defined on $\Omega$ by
$$
\psi[u_0](\cdot)=q(0^+,\cdot)(G^{-1})' (u_0(\cdot)),
$$
where $q$ solves the adjoint system \eqref{eq:q} associated to the control choice $u_0$.

Let $u_0^*$ be a solution to Problem~\eqref{prob:reduced}. 
Then, there exists \red{$\lambda\in [0,+\infty)$} such that 
\begin{equation}\label{eq:cond opt}
\begin{array}{ll}
\text{on }\{u_0^*=M\}, & \psi[u_0^*]\leq -\lambda,\\
\text{on }\{u_0^*=0\}, & \psi[u_0^*]\geq -\lambda,\\ 
\text{on }\{0<u_0<M\}, & \psi[u_0^*]= -\lambda,
\end{array}
\end{equation}
(called \red{necessary first order optimality condition}) or equivalently, the function $\Lambda$ defined by
$$
\Lambda : x\in \Omega \mapsto 
\min \{ u_0^*(x), \max \{ u_0^*(x)-M, \psi[u_0^*](x)+\lambda \}\}
$$
vanishes identically in $\Omega$. Moreover, one has $\lambda \left(\int_\Omega u_0^*(x)\, dx-C\right)=0$ (slackness condition).

Moreover, the second order optimality conditions for this problem read:
$d^2J_T(u_0^*)(h,h)\geq 0$ for every admissible perturbation $h$ such that $dJ_T(u_0^*)\cdot h=0$.
\end{proposition}
\begin{proof}
Let us introduce the Lagrangian functional associated to Problem~\eqref{prob:reduced}, given by
$$
\mathcal{L}:(u,\lambda)\in \mathcal{V}_{C,M}\times \RR_+\mapsto J_T(u)+\lambda \left(\int_{\Omega} u - C\right).
$$
According to Proposition~\ref{prop:adjoint}, and denoting by $d_{un}$ the differential operator with respect to the variable $u$, the Euler inequation associated to Problem~\eqref{prob:reduced} reads: $d_{un}\mathcal{L}(u,\lambda)\cdot h\geq 0$ for all admissible perturbation $h$ of $u_0^*$ in $\{u_0\in L^\infty(0,T), \ 0\leq u_0\leq M\text{ a.e. in }\Omega\}$. This can be rewritten
$$
\int_{\Omega} \left( \psi[u_0^*](x)+\lambda \right) h(x) \, dx\geq 0 
$$
for all functions $h$ as above. The analysis of such optimality condition is standard in optimal control theory (see for example \cite{MR1111666}) and yields:
$$
\text{on }\{u_0^*=M\}, \quad \psi[u_0^*]\leq -\lambda,\quad \text{ on }\{u_0^*=0\}, \quad \psi[u_0^*]\geq -\lambda,\quad 
\text{on }\{0<u_0<M\}, \quad \psi[u_0^*]= -\lambda.
$$
Moreover, one has $\lambda \left(\int_\Omega u_0^*(x)\, dx-C\right)=0$ (slackness condition). It remains to show that such conditions also rewrite $\Lambda(\cdot)=0$ in $\Omega$. It is straightforward that if the optimality conditions above are satisfied, then $\Lambda(\cdot)=0$ in $\Omega$. Let us examine the converse sense, assuming that $\Lambda(\cdot)=0$ in $\Omega$. Then, for a.e. $x\in \{u_0^*=0\}$, one has 
$$
\max \{ u_0^*(x)-M, \psi[u_0^*](x)+\lambda \}=\max \{-M, \psi[u_0^*](x)+\lambda \} \geq 0
$$
and thus,  $\psi[u_0^*](x)\geq -\lambda$. The analysis is exactly similar on the set $\{u_0^*=M\}$.
Finally, if $x$ denotes a Lebesgue point of the $\{0<u_0^*<M\}$, one has necessarily 
$$
\max \{ u_0^*(x)-M, \psi[u_0^*](x)+\lambda \}=0
$$
and therefore, $\psi[u_0^*](x)=-\lambda$. This concludes the first part of this proposition.
The second part is standard (see e.g. \cite{JBHU}).
\end{proof}
We infer from this result that either the pointwise or the integral constraint is saturated by every minimizer $u_0^*$.
\begin{corollary}\label{cor:0959}
Let $u_0^*$ be a solution to Problem~\eqref{prob:reduced}. Then, one has necessarily
$$
\int_\Omega u_0^*(x)\, dx=\min \{C,M|\Omega|\}.
$$
\end{corollary}
\begin{proof}
Let us first assume that $M\geq C/|\Omega|$. Let us argue by contradiction, assuming that $\int_\Omega u_0^* < C$. Let $p$ (resp. $q$) denote the solution to the direct problem \eqref{eq:psimple} (resp. the adjoint problem \eqref{eq:q}) associated to the control choice $u_0^*$. According to Theorem~\ref{theo:1orderopt} and its proof, the slackness condition implies that $\lambda=0$. Recall that one has $p(t,x)\in (0,1)$ for a.e. $(t,x)\in (0,T)\times \Omega$, as highlighted in Section~\ref{sec:OCP}, and therefore $q(T,\cdot)\in (-1,0)$ a.e. in $\Omega$. A simple comparison argument yields that $q$ is negative in $(0,T)\times \Omega$ \blue{(see e.g. \cite{conway1977comparison})}. Since $G$ is bijective and increasing, so is $G^{-1}$ and we infer that $\psi$ is negative in $\Omega$. By using Theorem~\ref{theo:1orderopt}, we get that necessarily, $u_0^*(\cdot)=M$, which is in contradiction with the assumption above on $M$ and $C$.

The case where $M<C/|\Omega|$ is solved hereafter, in the proof of Theorem~\ref{prop1.1}.
\end{proof}
\subsubsection{Optimality of constant solutions}\label{sec:optCstsol}

This section is devoted to the proof the our main results, that is Therem  \ref{prop1.1}. Let us first show $(i)$.
The proof rests upon a simple comparison argument: one shows more precisely that $\bar u_M$ solves Problem~\eqref{prob:reduced} as soon as it belongs to $\mathcal{V}_{C,M}$ which is equivalent to the condition above on the parameters.

Let $u\in \mathcal{V}_{C,M}$. 
Let $p$ and $p_M$ denote the solutions to System~\eqref{eq:psimple} corresponding respectively to the control choices $u$ and $\bar u_M$. 

Since $u$ belongs to $\mathcal{V}_{C,M}$ and $G^{-1}$ is increasing, one has $G^{-1}(u(x) )\leq G^{-1}(M)$ for a.e. $x \in \Omega$, meaning that $p(0^+,\cdot) \leq p_M(0^+,\cdot)
$ on $\Omega$. According to the parabolic comparison principle, we infer that $p(t,\cdot) \leq p_M(t,\cdot)$ on $\Omega$, for all $t \in [0,T)$, so that one gets in particular that $p^*(T,\cdot) \leq p_M(T,\cdot)$ in $\Omega$, and therefore, $J_T(\bar u_M) \leq J_T(u)$. Uniqueness follows from the monotonicity of $G$ and the comparison principle, since $0\leqslant u\leqslant M$ a.e. in $\Omega$.

Let us now prove $(ii)$. Set $c=C/|\Omega|$. According to the optimality conditions \eqref{eq:cond opt}, since $c<M$, the function $\bar u$ identically equal to the constant $c$ satisfies the first order optimality conditions if, and only if, there exists $\lambda\in \RR_+$ such that
$\psi(\cdot)=-\lambda$ in $\Omega$. Since $(G^{-1})' (\bar u(\cdot))$ is constant in $\Omega$, this is equivalent to say that ${q}(0^+,\cdot)$ is constant in $\Omega$.

First, observe that, by uniqueness of the solutions to the reaction-diffusion system \eqref{eq:psimple}, the associated solution $\bar{p}$ is constant in space. Moreover, writing $\bar{p}(t,\cdot)=\bar{p}(t)$ with a slight abuse of notation, one easily sees that $\bar{p}$ solves the ODE
\begin{equation}
\left\lbrace
\begin{array}{lll}
\bar{p}'(t) = f(\bar{p}(t)), & t\in[0,T],\\
\bar{p}(0^+)=G^{-1}(c). & 
\end{array}\right.
\label{p_bar}
\end{equation}
Standard uniqueness arguments coming from the Cauchy-Lipschitz theorem show that if $\bar{p}(0^+) \notin \{0,\theta,1\}$ (the set of roots of $f$), then $f(\bar{p}(\cdot))$ does not vanish on $[0,T]$ and has hence a constant sign.

Note that, since $c \neq 0$, one cannot have $\bar{p}(0^+)=0$. Similarly, noting that $G$ is an increasing bijection from $[0,1)$ into $[0,+\infty)$, we infer that one cannot have $\bar{p}(0^+)=1$.
Let $\bar{p}(0^+) \in (0,\theta)\cup (\theta,1)$. Then, $f(\bar{p}(0^+)) \neq 0$, and using that $\bar{p}$ has a constant sign, which allows us to write
$$
\bar{p}'(t)=f(\bar{p}(t))  \Rightarrow \forall t\in [0,T], \quad \int_{\bar{p}(0)}^{\bar{p}(t)} \frac{1}{f(u)} du = t
$$
and therefore,
$$
\bar{p}(t)=F^{-1}(t+F(G^{-1}(\bar u))),
$$
for all $t\in [0,T]$, where $F$ denotes an antiderivative of $1/f$. 
Indeed, since $f$ has a constant sign, $F$ is monotone and continuous, whence the existence of $F^{-1}$.

Proceeding similarly for the solution $\bar{q}$ to System~\eqref{eq:q} associated to $p=\bar{p}$ drives us to look for constant solutions with respect to the space variable.
Let $\bar{q}$ denote such a solution (whenever it exists). Hence, it solves
\begin{equation*}
\left\lbrace
\begin{array}{lll}
\bar{q}' (t) = -f'(\bar{p}(t)) \bar{q}(t), & t\in[0,T], \\
\bar{q} (T)=\bar{p}(T) - 1 &
\end{array}\right.
\end{equation*}
and therefore,
$$
\bar{q}(t)=(\bar{p}(T) - 1) \exp \left(  \int_t^T f'(\bar{p}(s)) ds \right)   .
$$
By uniqueness of the solution to \eqref{eq:q}, it follows that $\bar{q}$ solves \eqref{eq:q}. 

Now, if $\bar{p}(0^+) = \theta$, meaning that $\bar u = G(\theta)$, then $\bar{p}(\cdot) = \theta$ and one has $\bar{q}(t)=(\theta-1) e^{(T-t) f'(\theta)}$ for all $t\in [0,T]$.
 
All in all, we get that $\bar{q}(0^+,\cdot)$ is constant on $\Omega$ and the switching function $\psi$, which is constant, reads
$$
\psi(\cdot)=(G^{-1})'(c)(\bar{p}(T) - 1) \exp \left(  \int_0^T f'(\bar{p}(s)) ds \right)\leq 0,
$$
by using that $\bar{p}(t)\in (0,1)$ for all $t\in [0,T]$ and that $G$ is bijective and increasing.
We infer that the first order optimality conditions are satisfied by $\bar u$.

To investigate the second order optimality conditions, it is convenient to introduce the Hilbert basis $\{w_n\}_{n\in \NN^*}$ of $L^2(\Omega)$ made of the Neumann-Laplacian eigenfunctions defined by:
$$
w_1(\cdot)=\frac{1}{|\Omega|},\quad \text{and for $n\geq 2$, $w_n$ solves the p.d.e. } 
\left\lbrace
\begin{array}{lll}
-D \Delta w_n  = \lambda_n w_n,\quad \quad &\text{in } \Omega,\\
\partial_n w_n = 0,\quad \quad &\text{on }(0,T)\times \partial \Omega,\\
\int_\Omega w_n(x)\, dx=0, &\\
\|w_n\|_{L^2(\Omega)}=1, 
\end{array}\right.
$$ 
where $(\lambda_n)_{n\geq 2}$ denotes the sequence of associated positive eigenvalues.

In this setting, let us expand every admissible perturbation $h$ as
$$
h=\sum_{n=1}^{+\infty}\alpha_n w_n\quad \text{with }\alpha_n=\langle h,w_n\rangle_{L^2(\Omega)}\text{ for all }n\in \NN^*.
$$
Using that the solution $\bar p$ to \eqref{eq:psimple} does not depend on the space variable, it is standard to expand $\dot{p}$ as
\begin{equation*}
\dot{p}(t,x)=\sum_{n=1}^{\infty} \alpha_n v_n(t) w_n(x)\mbox{ for each }t\in(0,T),\quad x\in\Omega,
\end{equation*}
where $v_n$ solves the o.d.e. $v'_n(t) = \left( -\lambda_n + f'(\bar{p}(t)) \right) v_n(t)$ and $v_n(0)=(G^{-1})'(c)$ so that
$$
v_n(t)=(G^{-1})'(c) \exp \left( -\lambda_n t + \int_0^t f'(\bar{p}(s)) ds \right).
$$
According to Proposition~\ref{prop:adjoint}, one thus computes 
\begin{eqnarray*}
d^2J_T(\bar{u})(h,h) &=&\int_{\Omega}  \dot{p}(T,x)^2\, dx 
- \int_{\Omega}  \int_0^T   \dot{p}(t,x)^2 f''(\bar{p}(t))\bar{q}(t) \, dt dx \\
&& +\int_\Omega \bar q(0)(G^{-1})''(c)h(x)^2\, dx
\\
&=&\int_{\Omega} \left(\sum_{n=1}^{+\infty} \alpha_n v_n(T) w_n(x)\right)^2\, dx 
- \int_{\Omega}  \int_0^T   \left(\sum_{n=1}^{+\infty} \alpha_n v_n(t) w_n(x)\right)^2 f''(\bar{p}(t))\bar{q}(t) \, dt dx\\
&& +\int_\Omega \bar q(0)(G^{-1})''(c)\left(\sum_{n=1}^{+\infty}\alpha_nw_n(x)\right)^2\, dx.
\end{eqnarray*}
Using that  $\{w_n\}_{n\in \NN^*}$ is orthonormal in $L^2(\Omega)$, we finally get the following diagonalized expression of the second order derivative
$$
d^2J_T(\bar{u})(h,h)=\sum_{n=1}^{+\infty} \delta_n(T)\alpha_n^2\quad \text{with }\delta_n(T)= v_n(T)^2 - \int_0^T  f''(\bar{p}(t)) \bar{q}(t)   v_n(t)^2 dt +\bar q(0)(G^{-1})''(c).
$$
The signature of $d^2J_T(\bar{u})(h,h)$ seen as an infinite quadratic form with respect to $h$ is then directly given by the sign of the coefficients $\delta_n$. Notice that for all $n\in \NN^*$, one has
\begin{eqnarray*}
\delta_n(T) &=& v_n(T)^2+(\bar p(T)-1)e^{\int_0^Tf'(\bar p)}\left[(G^{-1})''(c)-\int_0^T f''(\bar p(t))e^{-\int_0^t f'(\bar p)}v_n(t)^2\, dt\right].
\end{eqnarray*}
Let us first assume that $C\leq |\Omega|G(\theta)$, meaning that $(G^{-1})(c)\leq \theta$. In that case, since $\bar p$ solves \eqref{p_bar}, and that the three roots of $f$ are 0, $\theta$ and 1, one infers that $\bar p$ is a decreasing function and that $f(\bar p)$ remains negative all along $(0,T)$. Furthermore, on $(0,\theta)$, the function $f''$ is positive. Finally, one computes
$(G^{-1})''(c)=(G^{-1})'(c)g'(G^{-1}(c))$ which is negative since so is $g'$ on $(0,1)$. Combining all these facts, we infer that
$$
(G^{-1})''(c)-\int_0^T f''(\bar p(t))e^{-\int_0^t f'(\bar p)}v_n(t)^2\, dt<0
$$
and since $\bar p(T)<1$, \blue{it} follows that 
$$
\delta_n(T)>(\bar p(T)-1)(G^{-1})''(c)e^{\int_0^Tf'(\bar p)} >0
$$ 
for every $n\in \NN^*$. \red{Therefore, by setting
$$
K_T=(\bar p(T)-1)(G^{-1})''(c)e^{\int_0^Tf'(\bar p)}>0,
$$
we get that for every admissible perturbation $h$, one has
$$
d^2J_T(\bar{u})(h,h) \geq K_T\sum_{n=1}^{+\infty}\alpha_n^2=K_T\Vert h \Vert_{L^2(\Omega)}^2.
$$
}
Expanding $J_T$ at the second order at $\bar u$, it is then standard that this condition implies that $\bar u$ is a local minimizer for the functional $J_T$.

\subsubsection{Constant solutions are not always global minimizers}\label{constant:nonglobal}

We recall the following well-known result (see \cite{OuyangShi}).
\begin{lemma}
  Let $\alpha\in (\theta_c,1)$. There exists a unique solution, denoted $v_\alpha$, of the Cauchy problem
  $$
  - \frac{d-1}{r} v'(r)  - v''(r) = f(v(r)), \text{ on } (0, +\infty), \quad v(0) = \alpha, \quad v'(0)= 0.
  $$
  Moreover, $\blue{ r\in(0,+\infty)} \mapsto v_\alpha(r)$ is decreasing, and there exists $R_\alpha >0$ such that $v_\alpha(R_\alpha) = 0$.
\end{lemma}
In other words, this lemma states the existence of radially symmetric steady-states to the stationary equation associated to \eqref{eq:psimple}.
We then deduce the existence of stationary subsolutions for System~\eqref{eq:psimple} that are positive and compactly supported, provided the domain contains a large enough ball, in other words that the inradius of $\Omega$ be large enough.
\begin{corollary}\label{subsol1}
Let us assume that a ball of radius $R_\alpha$  is compactly included in $\Omega$ for some $\alpha\in (\theta_c,1)$, in other words that there exists $O_\alpha\in \Omega$ such that $\overline{B(O_\alpha,R_\alpha)}\subset \Omega$. Then, $w_\alpha := \max \{0,v_\alpha(\|x-O_\alpha\|)\}$ is a subsolution of \eqref{eq:psimple} if, and only if $G(w_\alpha) \leq u_0$.
\end{corollary}

Using that $w_\alpha$ is a subsolution, we deduce the following comparison result.
\begin{corollary}\label{subsol2}
  For any $\alpha \in (\theta_c,1)$ such that $\Omega$ contains strictly a ball of radius $R_\alpha$, that is there exists $O_\alpha\in \Omega$ such that $\overline{B(O_\alpha,R_\alpha)}\subset \Omega$, and $G(w_\alpha)\leq u_0$, the solution of \eqref{eq:psimple} verifies $p(t,\cdot)\geq w_\alpha$ on $\Omega$ for any $t\geq 0$.
\end{corollary}
Let us introduce
$$
C_\alpha := \int_\Omega G(w_\alpha(x))\,dx.
$$

Notice that the family of subsolutions $(w_\alpha)_\alpha$ have already been used to provide a sufficient condition on the release function to initiate propagation of infected mosquitoes \cite{Zubelli}.
\begin{remark}
  It is worth mentioning that in the one dimensional case, the expressions for $R_\alpha$ and $C_\alpha$ are completely explicit:
  $$
  R_\alpha = \int_0^\alpha \frac{dw}{\sqrt{2(F(\alpha)-F(w))}}, \qquad
  C_\alpha = \int_0^\alpha \frac{2 G(w)\,dw}{\sqrt{2(F(\alpha)-F(w))}}.
  $$
\end{remark}

We are now in position to prove Proposition \ref{prop:cstpasoptim1} that we rewrite more precisely using the notations  above.
\begin{proposition}\label{prop:cstpasoptim}
Let us assume \eqref{eq:assumptionCLM}.
Assume moreover the existence of $\alpha\in (\theta_c,G^{-1}(M)]$ such that $\Omega$ contains strictly a ball of radius $R_\alpha$, and $C_\alpha \leq C$.
  Then the constant solution $\bar{u}:=\frac{C}{|\Omega|}$ is not a global minimizer of the optimization problem \ref{prob:reduced} whenever $T$ is large enough.
\end{proposition}
\begin{proof}
From assumption \eqref{eq:assumptionCLM}, we have $G^{-1}(\bar{u})<\theta$, hence we have already seen in Section~\ref{sec:optCstsol} that the solution, denoted $\bar{p}$, of \eqref{eq:psimple} with initial data $G^{-1}(\bar{u})$ is constant in space and decreasing with respect to time. More precisely, it solves the ODE
$$
\bar{p}' = f(\bar{p}), \qquad \bar{p}(0) = G^{-1}(\bar{u}) = G^{-1}\Big(\frac{C}{|\Omega|}\Big).
$$
Hence, when $t\to +\infty$, $\bar{p}(t)$ decays to $0$.

For any $\alpha\in (\theta_c,G^{-1}(M)]$ satisfying the assumptions above, the subsolution $w_\alpha$ defined in Corollary \ref{subsol1} is such that $G(w_\alpha)\in\calV_{C,M}$. From Corollary \ref{subsol2}, if we take $u_0\in \calV_{C,M}$ such that $G^{-1}(u_0)\geq w_\alpha$, then for all $t\geq 0$, the corresponding solution to \eqref{eq:psimple} verifies $p(t,\cdot)\geq w_\alpha$.
Hence $J_T(u_0) \leq \frac 12 \int_\Omega (1-w_\alpha(x))^2\,dx$.

  Moreover, since $\bar{p}(T)\to 0$ as $T\to + \infty$, we have that for $T$ large enough
  $$
  \int_\Omega (1-\bar{p}(T))^2 \,dx  = (1-\bar{p}(T))^2 |\Omega| > \int_\Omega (1-w_\alpha(x))^2\,dx.
  $$
  Hence, $\bar{u}$ is not a global minimum of $J_T$ at time $T$ since $J_T(u_0) < J_T(\bar{u}) = \frac 12 \int_\Omega (1-\bar{p}(T))^2\,dx$.
\end{proof}

\begin{remark}
  If $G^{-1}(M)>\theta_c$, if the inradius of $\Omega$ is large enough and \blue{if} $C$ is large enough, it is always possible to find $\alpha$ satisfying the assumptions of Proposition~\ref{prop:cstpasoptim}.
  For instance, it suffices to choose $\alpha=G^{-1}(M)$ and to take $C\geq C_{G^{-1}(M)}$ and the inradius of $\Omega$ large enough so that \eqref{eq:assumptionCLM} holds and a ball of radius $R_{G^{-1}(M)}$ be included in $\Omega$.
\end{remark}


\section{Numerical experiments}\label{sec:num}

In this section, we provide some numerical approximations of solutions for the optimal control problem \eqref{prob:reduced}. 

The parameter values are given in Tables \ref{tab:value} and \ref{tab:value2}. We will assume that $\Omega$ is an interval $(0,L)$, i.e. $d=1$.
From these tables, we deduce that $s_f = 0.1$, $\delta=\frac{4}{3}$, and thus $\theta = \frac{s_f+\delta -1}{\delta s_h}=\frac{13}{36}$. System \eqref{eq:psimple} will be discretized with an explicit Euler scheme in time and a standard finite difference approximation of the Laplacian. In all simulations, the number of steps in space and time will be fixed to $20$ and $200$ respectively (in order to satisfy the CFL condition). The solution of the optimal control problem will be obtained by testing and combining two approaches: 
\begin{itemize}
\item a Uzawa type algorithm, based on the gradient computation of Prop.~\ref{prop:adjoint}. It consists in alternating at each iteration a step of minimization of the Lagrangian associated with the problem with respect to the primal variable ($u_0$) and a step of maximization with respect to the Lagrange multiplier associated with the integral constraint. The minimization step is performed with a projected gradient type method, where $L^\infty$ constraints on $u_0$  are taken into account by means of a projection operator.  
\item the opensource optimization routine \textsc{GEKKO} (see \cite{beal2018gekko}) solving the optimization problem using the \textsc{IPOPT} (Interior Point OPTimizer) library, a software package for large-scale nonlinear problems by an interior-point filter line-search algorithm (see \cite{IPOPT}). This algorithm has been initialized with the previous control obtained by using the aforementioned Uzawa type algorithm.
\end{itemize}



\begin{table}[ht]
\centering
\begin{tabular}{|c|l|c|c|}
\hline
\textit{Parameter}&\textit{Name}&\textit{Value}\\\hline
 $F_{un}$& Normalized fecundity rate for uninfected mosquitoes &$1$ \\\hline
 $F_{in}$& Normalized fecundity rate for infected mosquitoes &$0.9$\\  \hline
     $d_{un}$&Death rate for uninfected mosquitoes&$0.27$ \\\hline
   $d_{in}$&Death rate for infected mosquitoes&$0.36$  \\ \hline
   $K$&Caring capacity& $0.06$ \\ \hline
    $s_h$&Cytoplasmic incompatibility&$0.9$  \\ \hline
  
\end{tabular}
\caption{Values of the parameters used in the simulations  (see \cite[sec. 2]{Zubelli}
)}
\label{tab:value}
\end{table}

\begin{table}[ht]
\centering
\begin{tabular}{|c|l|c|c|}
\hline
\textit{Parameter}&\textit{Name}&\textit{Value}\\\hline
 $T$&Time of experiment &$40$ \\\hline
 $D$& Diffusion coefficient &$1$\\  \hline
     $L=|\Omega|$&Size of the spatial domain&$30$ \\\hline
   \end{tabular}
\caption{Values of the parameters $T$, $D$ and $|\Omega|$ used in the simulations}
\label{tab:value2}
\end{table}

\blue{\begin{remark}\label{rk:assumpHfgver}
According to Remark~\ref{rmq:Hfg}, the assumption~\eqref{assump:fg} is satisfied for the particular choices of functions $f$ and $g$ given by \eqref{eq:fg} under \eqref{cond1} and \eqref{cond2} which hold true for the values of the parameters in Table~\ref{tab:value}. Furthermore, it is easy to check numerically that the assumption~\eqref{eq:defthetac} is satisfied for the values of the parameters taken from the case at hand (see e.g. \cite{Zubelli}). Indeed, for the parameter values, $f<0$ on $(0,\theta)$, which implies that $F<0$ on $(0,\theta)$, and moreover $F(1)>0$ (see Figure~\ref{zeros_fonction_f}).
\end{remark}
}

%

%
%
Let us distinguish between two cases:
\begin{center}
\textit{Case $C/|\Omega|>M$.}
\end{center}
In Figure \ref{fig:1}, the local minimizers of Problem~\eqref{prob:reduced}  for $C=1.2$ and $M=0.02$ (left) (resp. $M=0.03$ (right)) obtained by using the aformentioned Uzawa and Gekko algorithms are reported.
We observe the extinction (resp. the invasion) of the population.
One recovers the theoretical result stated in item (i) of Theorem~\ref{prop1.1}, in other words that the constant function equal to $M$ solves Problem~\eqref{prob:reduced} whenever $C\geq M|\Omega|$ (see Table \ref{tab:value3}).
In this situation, the space dependency has no impact on the time dynamics, i.e. the dynamics is the same as if there is no diffusion. Then, since it is a bistable dynamics, when $M < G(\theta)$ there is extinction of the population, whereas there is invasion when $M > G(\theta)$.

\begin{center}
\textit{Case $C/|\Omega|<M$.}
\end{center}
This situation is illustrated in Figure~\ref{fig:2} and \ref{fig:2bis} with Gekko algorithm and Figure \ref{fig:3} with Uzawa algorithm for $C\in\{0.5,0.8\}$ and $M\in\{0.04,0.4\}$. 
\blue{
We can see in Figure~\ref{fig:2} that when the number total of mosquitoes released is too low (when $C=0.5$), then the infected population decreased until the extinction of this population. On the contrary, if the number total of mosquitoes released is higher (when $C=0.8$), then we obtain an invasion of the infected mosquitoes.
}
The simulation with the Uzawa algorithm in Figure \ref{fig:3} recovers the fact that $C/L$ is a local minimizer for Problem~\eqref{prob:reduced}.
Indeed this algorithm seems to converges always to this constant solution. 
Nevertheless, it is not a global minimum since Gekko provides a better control as it is illustrated thanks to the values of $J_T(\bar{u})$ reported in Table \ref{tab:value3}. Moreover, we see on Figure \ref{fig:2} that invasion of the infected population seems to occurs whereas the infected population seems to go to extinction in Figure \ref{fig:3}.
This is also in concordance with the result stated in Proposition \ref{prop:cstpasoptim}.

\begin{table}[ht]
\centering
\begin{tabular}{|c|c|c|c|c|c|}
\hline
Case&\textit{Parameters}&\begin{tabular}{c}$J_T(\overline{u})$ with \\ Gekko\end{tabular} &\begin{tabular}{c}$J_T(\overline{u})$ with\\ Uzawa\end{tabular} &$J_T(M)$ &$J_T(C/L)$\\\hline\hline
 \multirow{2}{*}{$C/|\Omega|>M$}&$M=0.02$, $C=1.2$&14.7& 14.7&14.7&\\
&$M=0.03$, $C=1.2$&3.61e-2 &3.61e-2&3.61e-2&\\  \hline
 \multirow{4}{*}{$C/|\Omega|<M$}&$M=0.04$, $C=0.5$& 14.0&14.8&&14.8\\ 
&$M=0.04$, $C=0.8$&2.30 &12.7&&12.7\\  
&$M=0.08$, $C=0.5$&13.8&14.8&&14.8\\ 
&$M=0.08$, $C=0.8$&2.25&12.7&&12.7\\  \hline
   \end{tabular}
\caption{Values of local optima computed thanks to Gekko and Uzawa algorithms and theoretical local optima}
\label{tab:value3}
\end{table}

\begin{figure}[H] 
\begin{center}
\begin{tikzpicture}[thick,scale=0.65, every node/.style={scale=1.0}] \begin{axis}[xlabel=$x$,ylabel={$p$ for $M=0.02$},
xmin=0,xmax=30,
ymin=0,
ymin=0,ymax=1,
legend style={at={(0.8,1.1)},
anchor=north west,
 legend columns=2}]
\addplot[color=red]coordinates { 
(0.0,0.28050558887)
(1.4285714285714286,0.28050558887)
(2.857142857142857,0.28050558613)
(4.285714285714286,0.28050558872)
(5.714285714285714,0.28050558789)
(7.142857142857143,0.28050558838)
(8.571428571428571,0.28050558833)
(10.0,0.28050558799)
(11.428571428571429,0.28050558782)
(12.857142857142858,0.28050558795)
(14.285714285714286,0.28050558811)
(15.714285714285715,0.28050558811)
(17.142857142857142,0.28050558795)
(18.571428571428573,0.28050558782)
(20.0,0.28050558799)
(21.42857142857143,0.28050558833)
(22.857142857142858,0.28050558838)
(24.285714285714285,0.28050558789)
(25.714285714285715,0.28050558872)
(27.142857142857142,0.28050558613)
(28.571428571428573,0.28050558887)
(30.0,0.28050558887)
 };

\addplot[color=black!60!green]coordinates { 
(0.0,0.18940993834)
(1.4285714285714286,0.18940993834)
(2.857142857142857,0.18940993843)
(4.285714285714286,0.18940993861)
(5.714285714285714,0.18940993883)
(7.142857142857143,0.18940993905)
(8.571428571428571,0.18940993926)
(10.0,0.18940993942)
(11.428571428571429,0.18940993953)
(12.857142857142858,0.1894099396)
(14.285714285714286,0.18940993963)
(15.714285714285715,0.18940993963)
(17.142857142857142,0.1894099396)
(18.571428571428573,0.18940993953)
(20.0,0.18940993942)
(21.42857142857143,0.18940993926)
(22.857142857142858,0.18940993905)
(24.285714285714285,0.18940993883)
(25.714285714285715,0.18940993861)
(27.142857142857142,0.18940993843)
(28.571428571428573,0.18940993834)
(30.0,0.18940993834)
 };

\addplot[color=blue]coordinates { 
(0.0,0.091325658598)
(1.4285714285714286,0.091325658598)
(2.857142857142857,0.091325658682)
(4.285714285714286,0.09132565884)
(5.714285714285714,0.09132565905)
(7.142857142857143,0.091325659286)
(8.571428571428571,0.091325659524)
(10.0,0.09132565974)
(11.428571428571429,0.091325659916)
(12.857142857142858,0.091325660039)
(14.285714285714286,0.091325660102)
(15.714285714285715,0.091325660102)
(17.142857142857142,0.091325660039)
(18.571428571428573,0.091325659916)
(20.0,0.09132565974)
(21.42857142857143,0.091325659524)
(22.857142857142858,0.091325659286)
(24.285714285714285,0.09132565905)
(25.714285714285715,0.09132565884)
(27.142857142857142,0.091325658682)
(28.571428571428573,0.091325658598)
(30.0,0.091325658598)
 };

\addplot[color=brown]coordinates { 
(0.0,0.03344261122)
(1.4285714285714286,0.03344261122)
(2.857142857142857,0.03344261125)
(4.285714285714286,0.033442611307)
(5.714285714285714,0.033442611385)
(7.142857142857143,0.033442611475)
(8.571428571428571,0.033442611568)
(10.0,0.033442611656)
(11.428571428571429,0.033442611729)
(12.857142857142858,0.033442611782)
(14.285714285714286,0.033442611809)
(15.714285714285715,0.033442611809)
(17.142857142857142,0.033442611782)
(18.571428571428573,0.033442611729)
(20.0,0.033442611656)
(21.42857142857143,0.033442611568)
(22.857142857142858,0.033442611475)
(24.285714285714285,0.033442611385)
(25.714285714285715,0.033442611307)
(27.142857142857142,0.03344261125)
(28.571428571428573,0.03344261122)
(30.0,0.03344261122)
 };

\addplot[color=orange]coordinates { 
(0.0,0.010739647519)
(1.4285714285714286,0.010739647519)
(2.857142857142857,0.010739647526)
(4.285714285714286,0.010739647539)
(5.714285714285714,0.010739647557)
(7.142857142857143,0.010739647578)
(8.571428571428571,0.0107396476)
(10.0,0.010739647621)
(11.428571428571429,0.010739647639)
(12.857142857142858,0.010739647652)
(14.285714285714286,0.010739647659)
(15.714285714285715,0.010739647659)
(17.142857142857142,0.010739647652)
(18.571428571428573,0.010739647639)
(20.0,0.010739647621)
(21.42857142857143,0.0107396476)
(22.857142857142858,0.010739647578)
(24.285714285714285,0.010739647557)
(25.714285714285715,0.010739647539)
(27.142857142857142,0.010739647526)
(28.571428571428573,0.010739647519)
(30.0,0.010739647519)
 };

\addplot[color=black,dashed]coordinates { 
(0.0,0.28050547550804134)
(30.0,0.28050547550804134)
 };


\legend{$t=0$,$t=10$,$t=20$,$t=30$,$t=40$,$G^{-1}(M)$}
\end{axis} 
\end{tikzpicture}
\begin{tikzpicture}[thick,scale=0.65, every node/.style={scale=1.0}] \begin{axis}[xlabel=$x$,ylabel={$p$ for $M=0.03$},
xmin=0,xmax=30,
ymin=0,
ymin=0,ymax=1,
legend style={at={(0.8,1.1)},
anchor=north west,
 legend columns=2}]
\addplot[color=red]coordinates { 
(0.0,0.38223668027)
(1.4285714285714286,0.38223668027)
(2.857142857142857,0.38223668027)
(4.285714285714286,0.38223668027)
(5.714285714285714,0.38223668027)
(7.142857142857143,0.38223668027)
(8.571428571428571,0.38223668027)
(10.0,0.38223668027)
(11.428571428571429,0.38223668027)
(12.857142857142858,0.38223668027)
(14.285714285714286,0.38223668027)
(15.714285714285715,0.38223668027)
(17.142857142857142,0.38223668027)
(18.571428571428573,0.38223668027)
(20.0,0.38223668027)
(21.42857142857143,0.38223668027)
(22.857142857142858,0.38223668027)
(24.285714285714285,0.38223668027)
(25.714285714285715,0.38223668027)
(27.142857142857142,0.38223668027)
(28.571428571428573,0.38223668027)
(30.0,0.38223668027)
 };

\addplot[color=black!60!green]coordinates { 
(0.0,0.40953575179)
(1.4285714285714286,0.40953575179)
(2.857142857142857,0.40953575179)
(4.285714285714286,0.40953575179)
(5.714285714285714,0.40953575179)
(7.142857142857143,0.40953575179)
(8.571428571428571,0.40953575179)
(10.0,0.40953575179)
(11.428571428571429,0.40953575179)
(12.857142857142858,0.40953575179)
(14.285714285714286,0.40953575179)
(15.714285714285715,0.40953575179)
(17.142857142857142,0.40953575179)
(18.571428571428573,0.40953575179)
(20.0,0.40953575179)
(21.42857142857143,0.40953575179)
(22.857142857142858,0.40953575179)
(24.285714285714285,0.40953575179)
(25.714285714285715,0.40953575179)
(27.142857142857142,0.40953575179)
(28.571428571428573,0.40953575179)
(30.0,0.40953575179)
 };

\addplot[color=blue]coordinates { 
(0.0,0.49202511428)
(1.4285714285714286,0.49202511428)
(2.857142857142857,0.49202511428)
(4.285714285714286,0.49202511428)
(5.714285714285714,0.49202511428)
(7.142857142857143,0.49202511428)
(8.571428571428571,0.49202511428)
(10.0,0.49202511428)
(11.428571428571429,0.49202511427)
(12.857142857142858,0.49202511427)
(14.285714285714286,0.49202511427)
(15.714285714285715,0.49202511427)
(17.142857142857142,0.49202511427)
(18.571428571428573,0.49202511427)
(20.0,0.49202511428)
(21.42857142857143,0.49202511428)
(22.857142857142858,0.49202511428)
(24.285714285714285,0.49202511428)
(25.714285714285715,0.49202511428)
(27.142857142857142,0.49202511428)
(28.571428571428573,0.49202511428)
(30.0,0.49202511428)
 };

\addplot[color=brown]coordinates { 
(0.0,0.72501560418)
(1.4285714285714286,0.72501560418)
(2.857142857142857,0.72501560418)
(4.285714285714286,0.72501560417)
(5.714285714285714,0.72501560417)
(7.142857142857143,0.72501560417)
(8.571428571428571,0.72501560417)
(10.0,0.72501560416)
(11.428571428571429,0.72501560416)
(12.857142857142858,0.72501560416)
(14.285714285714286,0.72501560416)
(15.714285714285715,0.72501560416)
(17.142857142857142,0.72501560416)
(18.571428571428573,0.72501560416)
(20.0,0.72501560416)
(21.42857142857143,0.72501560417)
(22.857142857142858,0.72501560417)
(24.285714285714285,0.72501560417)
(25.714285714285715,0.72501560417)
(27.142857142857142,0.72501560418)
(28.571428571428573,0.72501560418)
(30.0,0.72501560418)
 };

\addplot[color=orange]coordinates { 
(0.0,0.95095286548)
(1.4285714285714286,0.95095286548)
(2.857142857142857,0.95095286548)
(4.285714285714286,0.95095286548)
(5.714285714285714,0.95095286547)
(7.142857142857143,0.95095286547)
(8.571428571428571,0.95095286546)
(10.0,0.95095286546)
(11.428571428571429,0.95095286545)
(12.857142857142858,0.95095286545)
(14.285714285714286,0.95095286545)
(15.714285714285715,0.95095286545)
(17.142857142857142,0.95095286545)
(18.571428571428573,0.95095286545)
(20.0,0.95095286546)
(21.42857142857143,0.95095286546)
(22.857142857142858,0.95095286547)
(24.285714285714285,0.95095286547)
(25.714285714285715,0.95095286548)
(27.142857142857142,0.95095286548)
(28.571428571428573,0.95095286548)
(30.0,0.95095286548)
 };

\addplot[color=black,dashed]coordinates { 
(0.0,0.3822365898168491)
(30.0,0.3822365898168491)
 };


\legend{$t=0$,$t=10$,$t=20$,$t=30$,$t=40$,$G^{-1}(M)$}
\end{axis} 
\end{tikzpicture}
\end{center}
\caption{
Case $C/|\Omega|>M$ : Optimal solution $p$ to Problem~\eqref{prob:reduced} at time $t\in\{0,10,20,30,40\}$ for $C=1.2$ and $M\in\{0.02,0.03\}$ thanks to Gekko and Uzawa algorithms
\label{fig:1}}
\end{figure}
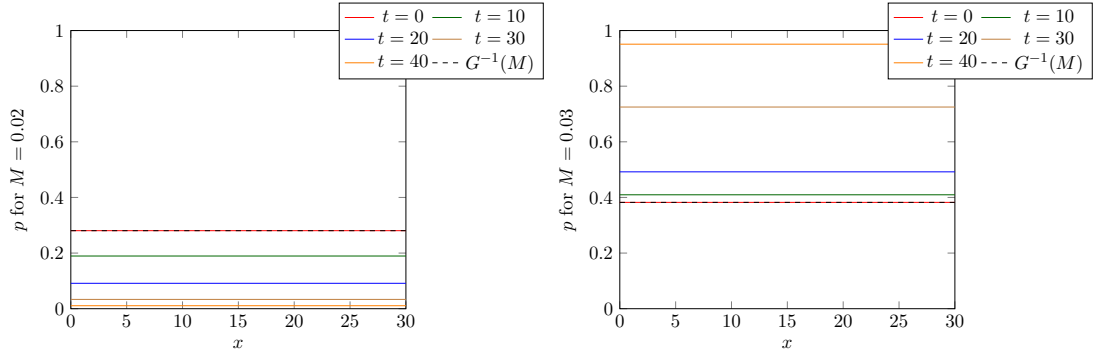

\begin{figure}[H] 
\begin{center}
\begin{tikzpicture}[thick,scale=0.65, every node/.style={scale=1.0}] \begin{axis}[xlabel=$x$,ylabel={$p$ for $M=0.04$ and $C=0.5$},
xmin=0,xmax=30,
ymin=0,
ymin=0,ymax=1,
legend style={at={(0.8,1.1)},
anchor=north west,
 legend columns=2}]
\addplot[color=red]coordinates { 
(0.0,0.0)
(1.4285714285714286,0.0)
(2.857142857142857,0.0)
(4.285714285714286,0.0)
(5.714285714285714,0.0)
(7.142857142857143,0.0)
(8.571428571428571,0.26128699681)
(10.0,0.43806603582)
(11.428571428571429,0.46300768059)
(12.857142857142858,0.4630076806)
(14.285714285714286,0.4630076806)
(15.714285714285715,0.4630076806)
(17.142857142857142,0.4630076806)
(18.571428571428573,0.46300768059)
(20.0,0.43806603582)
(21.42857142857143,0.26128699681)
(22.857142857142858,0.0)
(24.285714285714285,0.0)
(25.714285714285715,0.0)
(27.142857142857142,0.0)
(28.571428571428573,0.0)
(30.0,0.0)
 };

\addplot[color=black!60!green]coordinates { 
(0.0,0.019277135141)
(1.4285714285714286,0.019277135141)
(2.857142857142857,0.028041078272)
(4.285714285714286,0.047177975701)
(5.714285714285714,0.079530886148)
(7.142857142857143,0.12820965128)
(8.571428571428571,0.19462549177)
(10.0,0.27537536324)
(11.428571428571429,0.35944946778)
(12.857142857142858,0.42960777117)
(14.285714285714286,0.46931804935)
(15.714285714285715,0.46931804935)
(17.142857142857142,0.42960777117)
(18.571428571428573,0.35944946778)
(20.0,0.27537536324)
(21.42857142857143,0.19462549177)
(22.857142857142858,0.12820965128)
(24.285714285714285,0.079530886148)
(25.714285714285715,0.047177975701)
(27.142857142857142,0.028041078272)
(28.571428571428573,0.019277135141)
(30.0,0.019277135141)
 };

\addplot[color=blue]coordinates { 
(0.0,0.029018525411)
(1.4285714285714286,0.029018525411)
(2.857142857142857,0.035645105914)
(4.285714285714286,0.049717114787)
(5.714285714285714,0.072715004234)
(7.142857142857143,0.10632089706)
(8.571428571428571,0.15139634744)
(10.0,0.20622559135)
(11.428571428571429,0.26452049086)
(12.857142857142858,0.31499913551)
(14.285714285714286,0.34463175514)
(15.714285714285715,0.34463175514)
(17.142857142857142,0.31499913551)
(18.571428571428573,0.26452049086)
(20.0,0.20622559135)
(21.42857142857143,0.15139634744)
(22.857142857142858,0.10632089706)
(24.285714285714285,0.072715004234)
(25.714285714285715,0.049717114787)
(27.142857142857142,0.035645105914)
(28.571428571428573,0.029018525411)
(30.0,0.029018525411)
 };

\addplot[color=brown]coordinates { 
(0.0,0.024730584358)
(1.4285714285714286,0.024730584358)
(2.857142857142857,0.028634027394)
(4.285714285714286,0.03666572238)
(5.714285714285714,0.049157104348)
(7.142857142857143,0.066287326384)
(8.571428571428571,0.087669544022)
(10.0,0.11184687056)
(11.428571428571429,0.13593188386)
(12.857142857142858,0.15579717626)
(14.285714285714286,0.16713592788)
(15.714285714285715,0.16713592788)
(17.142857142857142,0.15579717626)
(18.571428571428573,0.13593188386)
(20.0,0.11184687056)
(21.42857142857143,0.087669544022)
(22.857142857142858,0.066287326384)
(24.285714285714285,0.049157104348)
(25.714285714285715,0.03666572238)
(27.142857142857142,0.028634027394)
(28.571428571428573,0.024730584358)
(30.0,0.024730584358)
 };

\addplot[color=orange]coordinates { 
(0.0,0.014984781981)
(1.4285714285714286,0.014984781981)
(2.857142857142857,0.016449146356)
(4.285714285714286,0.019342197211)
(5.714285714285714,0.023565001137)
(7.142857142857143,0.028907744319)
(8.571428571428571,0.034998789787)
(10.0,0.041276902999)
(11.428571428571429,0.047017557428)
(12.857142857142858,0.051433906285)
(14.285714285714286,0.05384245497)
(15.714285714285715,0.05384245497)
(17.142857142857142,0.051433906285)
(18.571428571428573,0.047017557428)
(20.0,0.041276902999)
(21.42857142857143,0.034998789787)
(22.857142857142858,0.028907744319)
(24.285714285714285,0.023565001137)
(25.714285714285715,0.019342197211)
(27.142857142857142,0.016449146356)
(28.571428571428573,0.014984781981)
(30.0,0.014984781981)
 };

\addplot[color=black,dashed]coordinates { 
(0.0,0.46300760876425756)
(30.0,0.46300760876425756)
 };

\addplot[color=black,dotted,very thick]coordinates { 
(0.0,0.24117424241095473)
(30.0,0.24117424241095473)
 };

\legend{$t=0$,$t=10$,$t=20$,$t=30$,$t=40$,$G^{-1}(M)$,$G^{-1}(C/L)$}
\end{axis} 
\end{tikzpicture}
\begin{tikzpicture}[thick,scale=0.65, every node/.style={scale=1.0}] \begin{axis}[xlabel=$x$,ylabel={$p$ for $M=0.04$ and $C=0.8$},
xmin=0,xmax=30,
ymin=0,
ymin=0,ymax=1,
legend style={at={(0.8,1.1)},
anchor=north west,
 legend columns=2}]
\addplot[color=red]coordinates { 
(0.0,0.31903432558)
(1.4285714285714286,0.31903432558)
(2.857142857142857,0.32444921351)
(4.285714285714286,0.33432582588)
(5.714285714285714,0.3474273746)
(7.142857142857143,0.36298752576)
(8.571428571428571,0.38098686269)
(10.0,0.40180303222)
(11.428571428571429,0.42548778873)
(12.857142857142858,0.45116774742)
(14.285714285714286,0.46300768056)
(15.714285714285715,0.46300768059)
(17.142857142857142,0.4630076806)
(18.571428571428573,0.4630076806)
(20.0,0.4630076806)
(21.42857142857143,0.46300768058)
(22.857142857142858,0.43165753334)
(24.285714285714285,0.30079427848)
(25.714285714285715,0.0)
(27.142857142857142,0.0)
(28.571428571428573,0.0)
(30.0,0.0)
 };

\addplot[color=black!60!green]coordinates { 
(0.0,0.30055979187)
(1.4285714285714286,0.30055979187)
(2.857142857142857,0.30962808162)
(4.285714285714286,0.32768759983)
(5.714285714285714,0.3543996318)
(7.142857142857143,0.38881928958)
(8.571428571428571,0.42899678361)
(10.0,0.47174823442)
(11.428571428571429,0.51282485388)
(12.857142857142858,0.54744700726)
(14.285714285714286,0.57081158019)
(15.714285714285715,0.57820750868)
(17.142857142857142,0.56494814248)
(18.571428571428573,0.52702675967)
(20.0,0.46350140413)
(21.42857142857143,0.38000009674)
(22.857142857142858,0.28940805602)
(24.285714285714285,0.20652436656)
(25.714285714285715,0.14146612925)
(27.142857142857142,0.09815434507)
(28.571428571428573,0.076818436451)
(30.0,0.076818436451)
 };

\addplot[color=blue]coordinates { 
(0.0,0.30632271178)
(1.4285714285714286,0.30632271178)
(2.857142857142857,0.32533793744)
(4.285714285714286,0.36327801112)
(5.714285714285714,0.41888556952)
(7.142857142857143,0.48804673762)
(8.571428571428571,0.56288639965)
(10.0,0.63330057579)
(11.428571428571429,0.69025551382)
(12.857142857142858,0.72777939913)
(14.285714285714286,0.74227856251)
(15.714285714285715,0.7307120112)
(17.142857142857142,0.68968268173)
(18.571428571428573,0.61702316358)
(20.0,0.51654508234)
(21.42857142857143,0.40242247376)
(22.857142857142858,0.29476818198)
(24.285714285714285,0.20811541524)
(25.714285714285715,0.1468060848)
(27.142857142857142,0.10890012295)
(28.571428571428573,0.091025170949)
(30.0,0.091025170949)
 };

\addplot[color=brown]coordinates { 
(0.0,0.42626598394)
(1.4285714285714286,0.42626598394)
(2.857142857142857,0.46355023047)
(4.285714285714286,0.53215239762)
(5.714285714285714,0.6190654375)
(7.142857142857143,0.70708301942)
(8.571428571428571,0.78222265715)
(10.0,0.83815577959)
(11.428571428571429,0.87442048592)
(12.857142857142858,0.89254669689)
(14.285714285714286,0.89320674935)
(15.714285714285715,0.87467282258)
(17.142857142857142,0.83222904556)
(18.571428571428573,0.75906788799)
(20.0,0.65073773775)
(21.42857142857143,0.51455569289)
(22.857142857142858,0.37491716103)
(24.285714285714285,0.25848550817)
(25.714285714285715,0.17648576937)
(27.142857142857142,0.12689302666)
(28.571428571428573,0.10399043108)
(30.0,0.10399043108)
 };

\addplot[color=orange]coordinates { 
(0.0,0.80714192856)
(1.4285714285714286,0.80714192856)
(2.857142857142857,0.82383460197)
(4.285714285714286,0.85176532905)
(5.714285714285714,0.8833528504)
(7.142857142857143,0.91244152687)
(8.571428571428571,0.93570755609)
(10.0,0.95212549411)
(11.428571428571429,0.96176184074)
(12.857142857142858,0.96469341893)
(14.285714285714286,0.96019576315)
(15.714285714285715,0.94605278494)
(17.142857142857142,0.91785568603)
(18.571428571428573,0.86847178323)
(20.0,0.78874785922)
(21.42857142857143,0.67220632409)
(22.857142857142858,0.52618288689)
(24.285714285714285,0.37892272376)
(25.714285714285715,0.26123709618)
(27.142857142857142,0.18533305359)
(28.571428571428573,0.14934561316)
(30.0,0.14934561316)
 };

\addplot[color=black,dashed]coordinates { 
(0.0,0.46300760876425756)
(30.0,0.46300760876425756)
 };

\addplot[color=black,dotted,very thick]coordinates { 
(0.0,0.3508840630147442)
(30.0,0.3508840630147442)
 };

\legend{$t=0$,$t=10$,$t=20$,$t=30$,$t=40$,$G^{-1}(M)$,$G^{-1}(C/L)$}
\end{axis} 
\end{tikzpicture}
\begin{tikzpicture}[thick,scale=0.65, every node/.style={scale=1.0}] \begin{axis}[xlabel=$x$,ylabel={$p$ for $M=0.08$ and $C=0.5$},
xmin=0,xmax=30,
ymin=0,
ymin=0,ymax=1,
legend style={at={(0.8,1.1)},
anchor=north west,
 legend columns=2}]
\addplot[color=red]coordinates { 
(0.0,0.0)
(1.4285714285714286,0.0)
(2.857142857142857,0.0)
(4.285714285714286,0.0)
(5.714285714285714,0.0)
(7.142857142857143,0.0)
(8.571428571428571,0.27554008026)
(10.0,0.43399043773)
(11.428571428571429,0.50187365235)
(12.857142857142858,0.53138640062)
(14.285714285714286,0.53904154754)
(15.714285714285715,0.5287227113)
(17.142857142857142,0.49524020135)
(18.571428571428573,0.41944327249)
(20.0,0.23870515043)
(21.42857142857143,0.0)
(22.857142857142858,0.0)
(24.285714285714285,0.0)
(25.714285714285715,0.0)
(27.142857142857142,0.0)
(28.571428571428573,0.0)
(30.0,0.0)
 };

\addplot[color=black!60!green]coordinates { 
(0.0,0.020921362456)
(1.4285714285714286,0.020921362456)
(2.857142857142857,0.030843511932)
(4.285714285714286,0.052815398283)
(5.714285714285714,0.090700301156)
(7.142857142857143,0.1488571109)
(8.571428571428571,0.22914285441)
(10.0,0.32571871112)
(11.428571428571429,0.42088761914)
(12.857142857142858,0.48970837836)
(14.285714285714286,0.51266466311)
(15.714285714285715,0.48348414952)
(17.142857142857142,0.41023892629)
(18.571428571428573,0.31370401904)
(20.0,0.21847996952)
(21.42857142857143,0.14071809786)
(22.857142857142858,0.084887943928)
(24.285714285714285,0.048368343177)
(25.714285714285715,0.02631372901)
(27.142857142857142,0.0142452425)
(28.571428571428573,0.0090069236388)
(30.0,0.0090069236388)
 };

\addplot[color=blue]coordinates { 
(0.0,0.034990781298)
(1.4285714285714286,0.034990781298)
(2.857142857142857,0.04354935109)
(4.285714285714286,0.061807831852)
(5.714285714285714,0.091762799392)
(7.142857142857143,0.13543715176)
(8.571428571428571,0.19309639034)
(10.0,0.26046232615)
(11.428571428571429,0.32644783933)
(12.857142857142858,0.37484228235)
(14.285714285714286,0.39122670714)
(15.714285714285715,0.37035722255)
(17.142857142857142,0.31880927467)
(18.571428571428573,0.25162845074)
(20.0,0.18464172685)
(21.42857142857143,0.12800453604)
(22.857142857142858,0.085141558475)
(24.285714285714285,0.055205458779)
(25.714285714285715,0.035737921136)
(27.142857142857142,0.024266690187)
(28.571428571428573,0.0189908221)
(30.0,0.0189908221)
 };

\addplot[color=brown]coordinates { 
(0.0,0.032557240783)
(1.4285714285714286,0.032557240783)
(2.857142857142857,0.038030531827)
(4.285714285714286,0.049309622074)
(5.714285714285714,0.066853906158)
(7.142857142857143,0.090796256693)
(8.571428571428571,0.12022677654)
(10.0,0.15238194083)
(11.428571428571429,0.18226581832)
(12.857142857142858,0.20347228521)
(14.285714285714286,0.21046447827)
(15.714285714285715,0.20121282584)
(17.142857142857142,0.1782722065)
(18.571428571428573,0.14741430769)
(20.0,0.11492107969)
(21.42857142857143,0.08544846799)
(22.857142857142858,0.061379038196)
(24.285714285714285,0.043283685255)
(25.714285714285715,0.030726213969)
(27.142857142857142,0.02293726844)
(28.571428571428573,0.019234233718)
(30.0,0.019234233718)
 };

\addplot[color=orange]coordinates { 
(0.0,0.021014510911)
(1.4285714285714286,0.021014510911)
(2.857142857142857,0.023181989629)
(4.285714285714286,0.027456319357)
(5.714285714285714,0.033665871053)
(7.142857142857143,0.041442176507)
(8.571428571428571,0.050132782905)
(10.0,0.058767553783)
(11.428571428571429,0.06613947663)
(12.857142857142858,0.071029513189)
(14.285714285714286,0.07252808862)
(15.714285714285715,0.070321808522)
(17.142857142857142,0.06479730615)
(18.571428571428573,0.056902881962)
(20.0,0.047844366719)
(21.42857142857143,0.038767437974)
(22.857142857142858,0.030549086128)
(24.285714285714285,0.023734278077)
(25.714285714285715,0.01858282186)
(27.142857142857142,0.015167211953)
(28.571428571428573,0.013472797054)
(30.0,0.013472797054)
 };

\addplot[color=black,dashed]coordinates { 
(0.0,0.6568582319083345)
(30.0,0.6568582319083345)
 };

\addplot[color=black,dotted,very thick]coordinates { 
(0.0,0.24117424241095473)
(30.0,0.24117424241095473)
 };

\legend{$t=0$,$t=10$,$t=20$,$t=30$,$t=40$,$G^{-1}(M)$,$G^{-1}(C/L)$}
\end{axis} 
\end{tikzpicture}
\begin{tikzpicture}[thick,scale=0.65, every node/.style={scale=1.0}] \begin{axis}[xlabel=$x$,ylabel={$p$ for $M=0.08$ and $C=0.8$},
xmin=0,xmax=30,
ymin=0,
ymin=0,ymax=1,
legend style={at={(0.8,1.1)},
anchor=north west,
 legend columns=2}]
\addplot[color=red]coordinates { 
(0.0,0.45746038749)
(1.4285714285714286,0.45746038749)
(2.857142857142857,0.46592362107)
(4.285714285714286,0.4773598257)
(5.714285714285714,0.48414030616)
(7.142857142857143,0.47845859188)
(8.571428571428571,0.45069199535)
(10.0,0.38352440783)
(11.428571428571429,0.23487502445)
(12.857142857142858,0.0)
(14.285714285714286,0.0)
(15.714285714285715,0.0)
(17.142857142857142,0.0)
(18.571428571428573,0.0)
(20.0,0.28529979195)
(21.42857142857143,0.40939531152)
(22.857142857142858,0.46505776645)
(24.285714285714285,0.48738223674)
(25.714285714285715,0.49203937169)
(27.142857142857142,0.48869643702)
(28.571428571428573,0.48472559147)
(30.0,0.48472559147)
 };

\addplot[color=black!60!green]coordinates { 
(0.0,0.64025360946)
(1.4285714285714286,0.64025360946)
(2.857142857142857,0.62545241994)
(4.285714285714286,0.59273826882)
(5.714285714285714,0.53810631834)
(7.142857142857143,0.46098295497)
(8.571428571428571,0.36857353328)
(10.0,0.2751764812)
(11.428571428571429,0.19511128782)
(12.857142857142858,0.13680192789)
(14.285714285714286,0.10294931298)
(15.714285714285715,0.093677020191)
(17.142857142857142,0.10905091859)
(18.571428571428573,0.1498474302)
(20.0,0.21639077091)
(21.42857142857143,0.30546080245)
(22.857142857142858,0.40677353104)
(24.285714285714285,0.50374974005)
(25.714285714285715,0.58123880419)
(27.142857142857142,0.63231984006)
(28.571428571428573,0.65699686969)
(30.0,0.65699686969)
 };

\addplot[color=blue]coordinates { 
(0.0,0.82886810562)
(1.4285714285714286,0.82886810562)
(2.857142857142857,0.80348977312)
(4.285714285714286,0.74970156581)
(5.714285714285714,0.6642565499)
(7.142857142857143,0.55016553459)
(8.571428571428571,0.42314753299)
(10.0,0.3065719791)
(11.428571428571429,0.21656238735)
(12.857142857142858,0.15694496123)
(14.285714285714286,0.1250755855)
(15.714285714285715,0.11787735742)
(17.142857142857142,0.13471249477)
(18.571428571428573,0.17794381694)
(20.0,0.25150601439)
(21.42857142857143,0.35638837803)
(22.857142857142858,0.4830678369)
(24.285714285714285,0.60885622119)
(25.714285714285715,0.71041796213)
(27.142857142857142,0.77726226364)
(28.571428571428573,0.80948818692)
(30.0,0.80948818692)
 };

\addplot[color=brown]coordinates { 
(0.0,0.92817721821)
(1.4285714285714286,0.92817721821)
(2.857142857142857,0.91090730321)
(4.285714285714286,0.87250799126)
(5.714285714285714,0.80603019506)
(7.142857142857143,0.70485018507)
(8.571428571428571,0.57147626689)
(10.0,0.42711918512)
(11.428571428571429,0.30268300129)
(12.857142857142858,0.21622175971)
(14.285714285714286,0.16977001669)
(15.714285714285715,0.16002942632)
(17.142857142857142,0.18590952345)
(18.571428571428573,0.25001212661)
(20.0,0.35445526793)
(21.42857142857143,0.49083876549)
(22.857142857142858,0.63315080545)
(24.285714285714285,0.75197485393)
(25.714285714285715,0.83482705063)
(27.142857142857142,0.88425592001)
(28.571428571428573,0.906832478)
(30.0,0.906832478)
 };

\addplot[color=orange]coordinates { 
(0.0,0.9675285306)
(1.4285714285714286,0.9675285306)
(2.857142857142857,0.95819055168)
(4.285714285714286,0.93700365066)
(5.714285714285714,0.8987927053)
(7.142857142857143,0.83607045258)
(8.571428571428571,0.74172708748)
(10.0,0.61680129952)
(11.428571428571429,0.48063057411)
(12.857142857142858,0.36589526587)
(14.285714285714286,0.2966295555)
(15.714285714285715,0.28122775295)
(17.142857142857142,0.3206541841)
(18.571428571428573,0.4115603404)
(20.0,0.53926842083)
(21.42857142857143,0.67364744164)
(22.857142857142858,0.78568149554)
(24.285714285714285,0.86481973123)
(25.714285714285715,0.91472213836)
(27.142857142857142,0.94291741755)
(28.571428571428573,0.95546078758)
(30.0,0.95546078758)
 };

\addplot[color=black,dashed]coordinates { 
(0.0,0.6568582319083345)
(30.0,0.6568582319083345)
 };

\addplot[color=black,dotted,very thick]coordinates { 
(0.0,0.3508840630147442)
(30.0,0.3508840630147442)
 };

\legend{$t=0$,$t=10$,$t=20$,$t=30$,$t=40$,$G^{-1}(M)$,$G^{-1}(C/L)$}
\end{axis} 
\end{tikzpicture}
\end{center}
\caption{Case {$C/|\Omega|<M$}: 
Optimal solution $p$ to Problem~\eqref{prob:reduced} at time $t\in\{0,10,20,30,40\}$ for  $C\in\{0.5,0.8\}$ and $M\in\{0.04,0.08\}$ thanks to Gekko algorithm.
\label{fig:2}}
\end{figure}
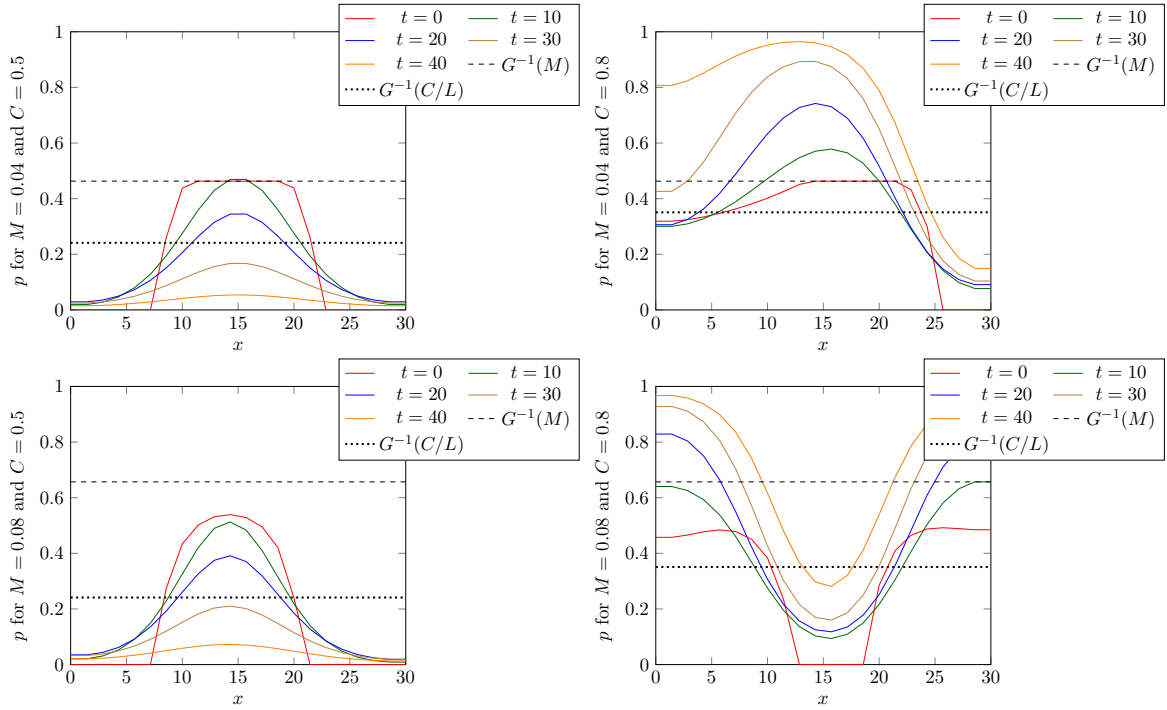

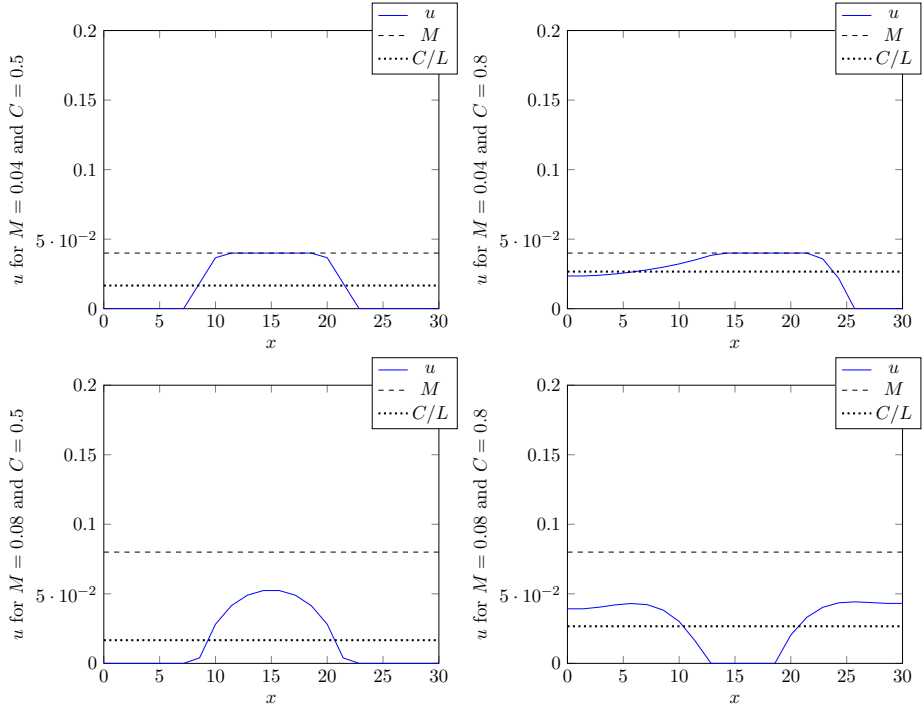
\begin{figure}[H] 
\begin{center}
\begin{tikzpicture}[thick,scale=0.65, every node/.style={scale=1.0}] \begin{axis}[xlabel=$x$,ylabel={$u$ for $M=0.04$ and $C=0.5$},
xmin=0,xmax=30,
ymin=0,
ymin=0,ymax=0.2,
legend style={at={(0.8,1.1)},
anchor=north west,
 legend columns=1}]
\addplot[color=blue]coordinates { 
(0.0,0.0)
(1.4285714285714286,0.0)
(2.857142857142857,0.0)
(4.285714285714286,0.0)
(5.714285714285714,0.0)
(7.142857142857143,0.0)
(8.571428571428571,0.018340804155)
(10.0,0.036659172343)
(11.428571428571429,0.04)
(12.857142857142858,0.04)
(14.285714285714286,0.04)
(15.714285714285715,0.04)
(17.142857142857142,0.04)
(18.571428571428573,0.04)
(20.0,0.036659172343)
(21.42857142857143,0.018340804155)
(22.857142857142858,0.0)
(24.285714285714285,0.0)
(25.714285714285715,0.0)
(27.142857142857142,0.0)
(28.571428571428573,0.0)
(30.0,0.0)
 };

\addplot[color=black,dashed]coordinates { 
(0.0,0.04)
(30.0,0.04)
 };

\addplot[color=black,dotted,very thick]coordinates { 
(0.0,0.01666666)
(30.0,0.01666666)
 };


\legend{$u$,$M$,$C/L$}
\end{axis} 
\end{tikzpicture}
\begin{tikzpicture}[thick,scale=0.65, every node/.style={scale=1.0}] \begin{axis}[xlabel=$x$,ylabel={$u$ for $M=0.04$ and $C=0.8$},
xmin=0,xmax=30,
ymin=0,
ymin=0,ymax=0.2,
legend style={at={(0.8,1.1)},
anchor=north west,
 legend columns=1}]
\addplot[color=blue]coordinates { 
(0.0,0.023523955248)
(1.4285714285714286,0.023523955248)
(2.857142857142857,0.024042547508)
(4.285714285714286,0.025004573324)
(5.714285714285714,0.026314398304)
(7.142857142857143,0.027923196469)
(8.571428571428571,0.029862088612)
(10.0,0.032218803949)
(11.428571428571429,0.035066840741)
(12.857142857142858,0.038382217003)
(14.285714285714286,0.04)
(15.714285714285715,0.04)
(17.142857142857142,0.04)
(18.571428571428573,0.04)
(20.0,0.04)
(21.42857142857143,0.04)
(22.857142857142858,0.035840592074)
(24.285714285714285,0.021820735574)
(25.714285714285715,0.0)
(27.142857142857142,0.0)
(28.571428571428573,0.0)
(30.0,0.0)
 };

\addplot[color=black,dashed]coordinates { 
(0.0,0.04)
(30.0,0.04)
 };
\addplot[color=black,dotted,very thick]coordinates { 
(0.0,0.02666666)
(30.0,0.02666666)
 };

\legend{$u$,$M$,$C/L$}
\end{axis} 
\end{tikzpicture}
\begin{tikzpicture}[thick,scale=0.65, every node/.style={scale=1.0}] \begin{axis}[xlabel=$x$,ylabel={$u$ for $M=0.08$ and $C=0.5$},
xmin=0,xmax=30,
ymin=0,
ymin=0,ymax=0.2,
legend style={at={(0.8,1.1)},
anchor=north west,
anchor=north west,
 legend columns=1}]
\addplot[color=blue]coordinates { 
(0.0,0.0)
(1.4285714285714286,0.0)
(2.857142857142857,0.0)
(4.285714285714286,0.0)
(5.714285714285714,0.0)
(7.142857142857143,0.0)
(8.571428571428571,0.0039361747357)
(10.0,0.028157724809)
(11.428571428571429,0.041534394507)
(12.857142857142858,0.049024208392)
(14.285714285714286,0.052347498763)
(15.714285714285715,0.052347498956)
(17.142857142857142,0.049024209052)
(18.571428571428573,0.041534395877)
(20.0,0.02815772736)
(21.42857142857143,0.0039361805428)
(22.857142857142858,0.0)
(24.285714285714285,0.0)
(25.714285714285715,0.0)
(27.142857142857142,0.0)
(28.571428571428573,0.0)
(30.0,0.0)
 };

\addplot[color=black,dashed]coordinates { 
(0.0,0.08)
(30.0,0.08)
 };
 
  \addplot[color=black,dotted,very thick]coordinates { 
(0.0, 0.01666666)
(30.0, 0.01666666)
 };
\legend{$u$,$M$,$C/L$}
\end{axis} 
\end{tikzpicture}
\begin{tikzpicture}[thick,scale=0.65, every node/.style={scale=1.0}] \begin{axis}[xlabel=$x$,ylabel={$u$ for $M=0.08$ and $C=0.8$},
xmin=0,xmax=30,
ymin=0,
ymin=0,ymax=0.2,
legend style={at={(0.8,1.1)},
anchor=north west,
 legend columns=1}]

\addplot[color=blue]coordinates { 
(0.0,0.039234595006)
(1.4285714285714286,0.039234595006)
(2.857142857142857,0.040407761886)
(4.285714285714286,0.042044350308)
(5.714285714285714,0.043044012049)
(7.142857142857143,0.042204824566)
(8.571428571428571,0.038318443728)
(10.0,0.030142577702)
(11.428571428571429,0.016154553568)
(12.857142857142858,0.0)
(14.285714285714286,0.0)
(15.714285714285715,0.0)
(17.142857142857142,0.0)
(18.571428571428573,0.0)
(20.0,0.020423599242)
(21.42857142857143,0.033111467242)
(22.857142857142858,0.040286289778)
(24.285714285714285,0.043530026066)
(25.714285714285715,0.044237583161)
(27.142857142857142,0.04372856138)
(28.571428571428573,0.043131364289)
(30.0,0.043131364289)
 };

\addplot[color=black,dashed]coordinates { 
(0.0,0.08)
(30.0,0.08)
 };

 \addplot[color=black,dotted,very thick]coordinates { 
(0.0, 0.02666666)
(30.0, 0.02666666)
 };

\legend{$u$,$M$,$C/L$}
\end{axis} 
\end{tikzpicture}
\end{center}
\caption{Case {$C/|\Omega|<M$}: 
Optimal control $u$ associated to the cases considered in Fig.~\ref{fig:2}, in other words solution of Problem~\eqref{prob:reduced} for $C\in\{0.5,0.8\}$ and $M\in\{0.04,0.08\}$ thanks to Gekko algorithm
\label{fig:2bis}}
\end{figure}

\begin{figure}[H] 
\begin{center}
\begin{tikzpicture}[thick,scale=0.65, every node/.style={scale=1.0}] \begin{axis}[xlabel=$x$,ylabel={$p$ for $M=0.04$ and $C=0.5$},
xmin=0,xmax=30,
ymin=0,
ymin=0,ymax=1,
legend style={at={(0.8,1.1)},
anchor=north west,
 legend columns=2}]
\addplot[color=red]coordinates { 
(0.0,0.24117424241095473)
(1.4285714285714286,0.24117424241095473)
(2.857142857142857,0.24117424241095473)
(4.285714285714286,0.24117424241095473)
(5.714285714285714,0.24117424241095473)
(7.142857142857143,0.24117424241095473)
(8.571428571428571,0.24117424241095473)
(10.0,0.24117424241095473)
(11.428571428571429,0.24117424241095473)
(12.857142857142858,0.24117424241095473)
(14.285714285714286,0.24117424241095473)
(15.714285714285715,0.24117424241095473)
(17.142857142857142,0.24117424241095473)
(18.571428571428573,0.24117424241095473)
(20.0,0.24117424241095473)
(21.42857142857143,0.24117424241095473)
(22.857142857142858,0.24117424241095473)
(24.285714285714285,0.24117424241095473)
(25.714285714285715,0.24117424241095473)
(27.142857142857142,0.24117424241095473)
(28.571428571428573,0.24117424241095473)
(30.0,0.24117424241095473)
 };

\addplot[color=black!60!green]coordinates { 
(0.0,0.13925807137764365)
(1.4285714285714286,0.13925807137764365)
(2.857142857142857,0.13925807137764365)
(4.285714285714286,0.13925807137764365)
(5.714285714285714,0.13925807137764365)
(7.142857142857143,0.13925807137764365)
(8.571428571428571,0.13925807137764365)
(10.0,0.13925807137764365)
(11.428571428571429,0.13925807137764365)
(12.857142857142858,0.13925807137764365)
(14.285714285714286,0.13925807137764365)
(15.714285714285715,0.13925807137764365)
(17.142857142857142,0.13925807137764365)
(18.571428571428573,0.13925807137764365)
(20.0,0.13925807137764365)
(21.42857142857143,0.13925807137764365)
(22.857142857142858,0.13925807137764365)
(24.285714285714285,0.13925807137764365)
(25.714285714285715,0.13925807137764365)
(27.142857142857142,0.1392580713776436)
(28.571428571428573,0.13925807137764357)
(30.0,0.13925807137764357)
 };

\addplot[color=blue]coordinates { 
(0.0,0.05778121623300752)
(1.4285714285714286,0.05778121623300752)
(2.857142857142857,0.05778121623300752)
(4.285714285714286,0.05778121623300752)
(5.714285714285714,0.05778121623300752)
(7.142857142857143,0.05778121623300752)
(8.571428571428571,0.05778121623300752)
(10.0,0.05778121623300752)
(11.428571428571429,0.05778121623300752)
(12.857142857142858,0.05778121623300752)
(14.285714285714286,0.05778121623300752)
(15.714285714285715,0.05778121623300752)
(17.142857142857142,0.05778121623300752)
(18.571428571428573,0.05778121623300752)
(20.0,0.05778121623300752)
(21.42857142857143,0.05778121623300752)
(22.857142857142858,0.05778121623300752)
(24.285714285714285,0.05778121623300752)
(25.714285714285715,0.05778121623300752)
(27.142857142857142,0.057781216233007505)
(28.571428571428573,0.057781216233007485)
(30.0,0.05778121623300748)
 };

\addplot[color=brown]coordinates { 
(0.0,0.01956318240564086)
(1.4285714285714286,0.01956318240564086)
(2.857142857142857,0.019563182405640863)
(4.285714285714286,0.019563182405640863)
(5.714285714285714,0.019563182405640863)
(7.142857142857143,0.019563182405640863)
(8.571428571428571,0.019563182405640863)
(10.0,0.019563182405640863)
(11.428571428571429,0.019563182405640863)
(12.857142857142858,0.019563182405640863)
(14.285714285714286,0.019563182405640863)
(15.714285714285715,0.019563182405640863)
(17.142857142857142,0.019563182405640863)
(18.571428571428573,0.019563182405640863)
(20.0,0.01956318240564086)
(21.42857142857143,0.01956318240564086)
(22.857142857142858,0.01956318240564085)
(24.285714285714285,0.019563182405640846)
(25.714285714285715,0.01956318240564084)
(27.142857142857142,0.01956318240564084)
(28.571428571428573,0.01956318240564083)
(30.0,0.01956318240564083)
 };

\addplot[color=orange]coordinates { 
(0.0,0.006104234692947273)
(1.4285714285714286,0.006104234692947273)
(2.857142857142857,0.006104234692947273)
(4.285714285714286,0.006104234692947273)
(5.714285714285714,0.006104234692947273)
(7.142857142857143,0.006104234692947273)
(8.571428571428571,0.006104234692947273)
(10.0,0.006104234692947273)
(11.428571428571429,0.006104234692947273)
(12.857142857142858,0.006104234692947273)
(14.285714285714286,0.006104234692947273)
(15.714285714285715,0.0061042346929472725)
(17.142857142857142,0.0061042346929472725)
(18.571428571428573,0.0061042346929472725)
(20.0,0.0061042346929472725)
(21.42857142857143,0.006104234692947271)
(22.857142857142858,0.006104234692947268)
(24.285714285714285,0.006104234692947267)
(25.714285714285715,0.006104234692947264)
(27.142857142857142,0.006104234692947262)
(28.571428571428573,0.006104234692947261)
(30.0,0.006104234692947261)
 };

\addplot[color=black,dashed]coordinates { 
(0.0,0.9298555024404189)
(30.0,0.9298555024404189)
 };

\addplot[color=black,dotted,very thick]coordinates { 
(0.0,0.24117424241095473)
(30.0,0.24117424241095473)
 };


\legend{$t=0$,$t=10$,$t=20$,$t=30$,$t=40$,$G^{-1}(M)$,$G^{-1}(C/L)$}
\end{axis} 
\end{tikzpicture}
\begin{tikzpicture}[thick,scale=0.65, every node/.style={scale=1.0}] \begin{axis}[xlabel=$x$,ylabel={$p$ for $M=0.04$ and $C=0.8$},
xmin=0,xmax=30,
ymin=0,
ymin=0,ymax=1,
legend style={at={(0.8,1.1)},
anchor=north west,
 legend columns=2}]
\addplot[color=red]coordinates { 
(0.0,0.3508840630147442)
(1.4285714285714286,0.3508840630147442)
(2.857142857142857,0.3508840630147442)
(4.285714285714286,0.3508840630147442)
(5.714285714285714,0.3508840630147442)
(7.142857142857143,0.3508840630147442)
(8.571428571428571,0.3508840630147442)
(10.0,0.3508840630147442)
(11.428571428571429,0.3508840630147442)
(12.857142857142858,0.3508840630147442)
(14.285714285714286,0.3508840630147442)
(15.714285714285715,0.3508840630147442)
(17.142857142857142,0.3508840630147442)
(18.571428571428573,0.3508840630147442)
(20.0,0.3508840630147442)
(21.42857142857143,0.3508840630147442)
(22.857142857142858,0.3508840630147442)
(24.285714285714285,0.3508840630147442)
(25.714285714285715,0.3508840630147442)
(27.142857142857142,0.3508840630147442)
(28.571428571428573,0.3508840630147442)
(30.0,0.3508840630147442)
 };

\addplot[color=black!60!green]coordinates { 
(0.0,0.32539950019264957)
(1.4285714285714286,0.32539950019264957)
(2.857142857142857,0.3253995001926497)
(4.285714285714286,0.3253995001926497)
(5.714285714285714,0.3253995001926498)
(7.142857142857143,0.3253995001926498)
(8.571428571428571,0.3253995001926498)
(10.0,0.3253995001926498)
(11.428571428571429,0.3253995001926498)
(12.857142857142858,0.3253995001926498)
(14.285714285714286,0.3253995001926498)
(15.714285714285715,0.3253995001926498)
(17.142857142857142,0.3253995001926498)
(18.571428571428573,0.3253995001926497)
(20.0,0.3253995001926497)
(21.42857142857143,0.32539950019264957)
(22.857142857142858,0.32539950019264957)
(24.285714285714285,0.32539950019264946)
(25.714285714285715,0.32539950019264946)
(27.142857142857142,0.32539950019264946)
(28.571428571428573,0.32539950019264946)
(30.0,0.32539950019264946)
 };

\addplot[color=blue]coordinates { 
(0.0,0.26796763826499964)
(1.4285714285714286,0.26796763826499964)
(2.857142857142857,0.26796763826499964)
(4.285714285714286,0.26796763826499964)
(5.714285714285714,0.26796763826499964)
(7.142857142857143,0.26796763826499964)
(8.571428571428571,0.26796763826499964)
(10.0,0.26796763826499964)
(11.428571428571429,0.26796763826499964)
(12.857142857142858,0.26796763826499964)
(14.285714285714286,0.26796763826499964)
(15.714285714285715,0.26796763826499964)
(17.142857142857142,0.26796763826499964)
(18.571428571428573,0.26796763826499964)
(20.0,0.26796763826499964)
(21.42857142857143,0.2679676382649996)
(22.857142857142858,0.2679676382649995)
(24.285714285714285,0.2679676382649994)
(25.714285714285715,0.2679676382649994)
(27.142857142857142,0.2679676382649994)
(28.571428571428573,0.2679676382649994)
(30.0,0.2679676382649994)
 };

\addplot[color=brown]coordinates { 
(0.0,0.1718136276547343)
(1.4285714285714286,0.17181362765473432)
(2.857142857142857,0.17181362765473432)
(4.285714285714286,0.17181362765473432)
(5.714285714285714,0.17181362765473432)
(7.142857142857143,0.17181362765473432)
(8.571428571428571,0.17181362765473432)
(10.0,0.17181362765473432)
(11.428571428571429,0.17181362765473432)
(12.857142857142858,0.17181362765473432)
(14.285714285714286,0.17181362765473432)
(15.714285714285715,0.17181362765473432)
(17.142857142857142,0.17181362765473432)
(18.571428571428573,0.17181362765473432)
(20.0,0.17181362765473432)
(21.42857142857143,0.17181362765473432)
(22.857142857142858,0.17181362765473432)
(24.285714285714285,0.17181362765473432)
(25.714285714285715,0.1718136276547343)
(27.142857142857142,0.17181362765473424)
(28.571428571428573,0.1718136276547342)
(30.0,0.1718136276547342)
 };

\addplot[color=orange]coordinates { 
(0.0,0.07839149740952306)
(1.4285714285714286,0.07839149740952307)
(2.857142857142857,0.07839149740952307)
(4.285714285714286,0.07839149740952307)
(5.714285714285714,0.07839149740952307)
(7.142857142857143,0.07839149740952307)
(8.571428571428571,0.07839149740952307)
(10.0,0.07839149740952307)
(11.428571428571429,0.07839149740952307)
(12.857142857142858,0.07839149740952307)
(14.285714285714286,0.07839149740952307)
(15.714285714285715,0.07839149740952307)
(17.142857142857142,0.07839149740952307)
(18.571428571428573,0.07839149740952307)
(20.0,0.07839149740952307)
(21.42857142857143,0.07839149740952314)
(22.857142857142858,0.07839149740952314)
(24.285714285714285,0.07839149740952314)
(25.714285714285715,0.07839149740952314)
(27.142857142857142,0.07839149740952314)
(28.571428571428573,0.07839149740952314)
(30.0,0.07839149740952314)
 };

\addplot[color=black,dashed]coordinates { 
(0.0,0.9298555024404189)
(30.0,0.9298555024404189)
 };

\addplot[color=black,dotted,very thick]coordinates { 
(0.0,0.3508840630147442)
(30.0,0.3508840630147442)
 };


\legend{$t=0$,$t=10$,$t=20$,$t=30$,$t=40$,$G^{-1}(M)$,$G^{-1}(C/L)$}
\end{axis} 
\end{tikzpicture}
\end{center}
\caption{Case {$C/|\Omega|<M$}: 
Optimal solution $p$ to Problem~\eqref{prob:reduced} at time $t\in\{0,10,20,30,40\}$ for  $C\in\{0.5,0.8\}$ and $M=0.04$  thanks to Uzawa algorithms
\label{fig:3}}
\end{figure}
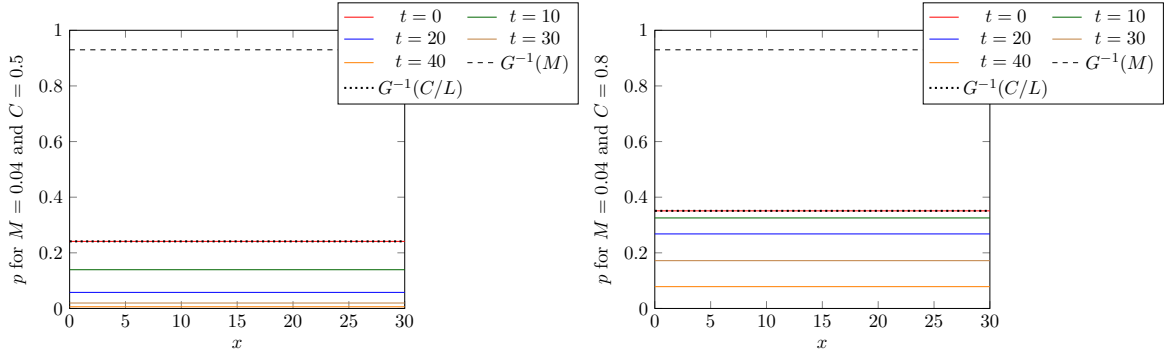

\section{Perspectives}
In a near future, we foresee to investigate a more involved model, closer to practical experiments, where one aims at determining release distributions in time and space, assuming that:
\begin{itemize}
\item releases are done periodically in time (for instance every week) and are impulses in time\footnote{We consider Dirac measures since at the time-level of the study (namely, some generations), the release can be considered as instantaneous.};
\item at each release, the largest allowed amount of mosquitoes is released, corresponding to the maximal  production capacity per week (which is relevant, according to the comparison principle).
\end{itemize}
As a consequence, we will be interested in determining the optimal way of releasing spatially the infected mosquitoes.
Considering $N$ releases, we denote by $t_0=0 < t_1 < \ldots <t_{N-1} < T$, $t_i=i \Delta T$, the release times. Rewriting the $L^1$ constraint on the control as $\langle u,1\rangle_{\mathcal{D'},\mathcal{D}((0,T)\times \Omega)}\leq C$, the control function reads
$$
u(t,x) = \sum_{i=0}^{N-1} u_i(x) \delta_{\{t=t_i\}}, \quad \mbox{ with } \sum_{i=0}^{N-1}\int_\Omega u_i(x)\,dx \leq C,
$$
where the pointwise constraint is modified into $0\leq u_i(\cdot)\leq M$.

The new optimal design problem reads
\[
\inf_{\mathbf{u}\in \mathcal{V}_{C,M}} \tilde{J}_T(\mathbf{u})
\tag{$\mathcal{P}'_{\text{full}}$}, \quad \text{where }\mathbf{u}=(u_i)_{0\leq i\leq N-1}, \quad \tilde{J}_T(\mathbf{u})=J_T\left(\sum_{i=0}^{N-1} u_i(\cdot) \delta_{\{t=t_i\}}\right)
\]
and
\begin{multline*}
\mathcal{V}_{C,M} \\= \left\{\mathbf{u} = (u_i(\cdot))_{0\leq i\leq N-1}, \quad 0\leq u_i\leq M \text{ a.e. in }\Omega, \ i\in \left\{0,\dots,N-1\right\}, \ \sum_{i=0}^{N-1} \int_\Omega u_i(x)\,dx \leq C\right\}.
\end{multline*}
As done in this article, System \eqref{eq:pintro} can be recast without source measure terms, coming from the specific form of the control functions. 

In a second time, we will also look at dropping the assumption on the frequency of releases and determine optimal times of releases (in the spirit of \cite{APSV2018}, where a simpler ODE model were considered).

Another interesting question is also raised by the spatial heterogeneities. Indeed, in field experiments the environment is not homogeneous in space. Then an important issue, from an experimental point of view, is to determine how to adapt the releases with respect to the spacial heterogeneities to optimize the success of the replacement strategies.

\section*{Acknowledgments}
M. Duprez, Y. Privat and N. Vauchelet were partially supported by the Project ''Analysis and simulation of optimal shapes - application to lifesciences'' of the Paris City Hall. Y. Privat was partially supported by the ANR Project ANR-18-CE40-0013 - SHAPO on Shape Optimization.
\blue{We also warmly thank the referees for their comments and suggestions.}
\bibliographystyle{abbrv}
\bibliography{biblio}

\end{document}